\documentclass{amsart}
\usepackage{amssymb}
\newtheorem{theorem}{Theorem}[section]
\newtheorem{lemma}[theorem]{Lemma}
\newtheorem{remark}[theorem]{Remark}
\newtheorem{definition}[theorem]{Definition}

\newtheorem{corollary}[theorem]{Corollary}
\newtheorem{proposition}[theorem]{Proposition}
\newtheorem{lem-def}[theorem]{Lemma-Definition}
\newcommand{\R}{\mathbb R}

\newcommand{\Z}{\mathbb Z}
\newcommand{\Q}{\mathbb Q}

\newcommand{\F}{\mathbb F}
\newcommand{\st}[1]{\vskip 1mm \noindent\makebox[4mm][r]{\bf #1} \hspace{1mm}}
\newcommand{\stst}[1]{\vskip 1mm \noindent\makebox[4mm][r]{\bf #1} \hspace{6mm}}
\newcommand{\ststst}[1]{\vskip 1mm \noindent\makebox[4mm][r]{\bf #1} \hspace{11mm}}
\newcommand{\stststst}[1]{\vskip 1mm \noindent\makebox[4mm][r]{\bf #1} \hspace{16mm}}
\newcommand{\ststststst}[1]{\vskip 1mm \noindent\makebox[4mm][r]{\bf #1} \hspace{21mm}}

\def\op{\operatorname}

\def\as#1{\renewcommand\arraystretch{#1}}

\def\bb{{\mathcal B}}

\def\df{\operatorname{Diff}}

\def\dsc{\operatorname{Disc}}
\def\diso{\lower.4ex\hbox{$\downarrow$}\raise.4ex\hbox{\mbox{\scriptsize
$\wr$}}}

\def\iso{\ \lower.3ex\hbox{\as{.08}$\begin{array}{c}\lra\\\mbox{\tiny $\sim\,$}\end{array}$}\ }

\def\j{\mathbf{j}}

\def\kb{\overline{k}}

\def\ks{k^{\op{sep}}}

\def\la{\lambda}
\def\lg{l\raise.6ex\hbox to.2em{\hss.\hss}l}
\def\ll{\mathcal{L}}
\def\lm{\lambda_{\operatorname{max}}}
\def\lmn{\lambda_{\operatorname{min}}}
\def\ln{\lambda_{\operatorname{next}}}
\def\lra{\longrightarrow}
\def\m{{\mathfrak m}}
\def\md#1{\; \mbox{\rm(mod }{#1})}

\def\n{\mathbf{n}}
\def\nm{\mathbf{m}}
\def\nn{\mathcal{N}}

\def\om{\omega}
\def\oo{\mathcal{O}}
\def\orb{\hbox to  .3em{$\backslash$}\backslash}
\def\ord{\op{ord}}
\def\p{\mathfrak{p}}
\def\P{\mathfrak{P}}

\def\res{\operatorname{Res}}

\def\rr{\mathcal{R}}

\def\sm{S_{\operatorname{max}}}
\def\smn{S_{\operatorname{min}}}
\def\sn{S_{\operatorname{next}}}

\def\t{\theta}

\def\tcal{\mathcal{T}}
\def\ty{\mathbf{t}}

\def\vv{{\mathcal V}}

\newcounter{cs}
\stepcounter{cs}
\newcommand{\casos}{\begin{itemize}}
\newcommand{\fcasos}{\end{itemize}\setcounter{cs}{1}}

\newfont{\tit}{cmr12 scaled \magstep3}

\title[Complexity of OM factorizations]{Complexity of OM factorizations of polynomials over local fields}

\author[Bauch]{Jens-Dietrich Bauch}
\address{Departament de Matem\`{a}tiques,
         Universitat Aut\`{o}noma de Barcelona,
         Edifici C, E-08193 Bellaterra, Barcelona, Catalonia, Spain}
\email{bauch@mat.uab.cat, nart@mat.uab.cat, hds@mat.uab.cat}
\author[Nart]{Enric Nart}
\author[Stainsby]{Hayden D. Stainsby}
\thanks{Partially supported by MTM2009-10359 from the
Spanish MEC}
\date{}

\makeatletter
\@namedef{subjclassname@2010}{%
  \textup{2010} Mathematics Subject Classification}

\begin{document}

\begin{abstract}
Let $k$ be a locally compact complete field with respect to a discrete valuation $v$. Let $\oo$ be the valuation ring, $\m$ the maximal ideal and $F(x)\in\oo[x]$ a monic separable polynomial of degree $n$. Let $\delta=v(\dsc(F))$. The Montes algorithm  computes an OM factorization of $F$. The single-factor lifting algorithm derives from this data a factorization of $F \md{\m^\nu}$, for a prescribed precision $\nu$. In this paper we find a new estimate for the complexity of the Montes algorithm, leading to an estimation of $O\left(n^{2+\epsilon}+n^{1+\epsilon}\delta^{2+\epsilon}+n^2\nu^{1+\epsilon}\right)$ word operations for the complexity of the computation of a factorization of $F \md{\m^\nu}$, assuming that the residue field of $k$ is small. 
\end{abstract}

\maketitle
\section*{Introduction}
Let $A$ be a Dedekind domain whose field of fractions $K$ is a global field. Let $L/K$ be a finite separable extension and $B$ the integral closure of $A$ in $L$. Let $\t\in L$ be a primitive element of $L/K$, with minimal polynomial $f(x)\in A[x]$.

Let $\p$ be a non-zero prime ideal of $A$, $v_\p$ the canonical $\p$-adic valuation, $K_\p$ the completion of $K$ at $\p$, and $\oo_\p$ the valuation ring of $K_\p$.

The Montes algorithm \cite{algorithm,HN} computes an \emph{OM representation} of every prime ideal $\P$ of $B$ lying over $\p$ \cite{newapp}. This algorithm carries out a program suggested by \O. Ore \cite{ore1,ore2}, and developed by S. MacLane in the context of valuation theory \cite{mcla,mclb}. An OM representation is a computational object supporting several data and operators, linked to one of the irreducible factors (say) $F(x)$ of $f(x)$ in $\oo_\p[x]$. Among these data, the OM representation contains all the \emph{Okutsu invariants} of $F$, which reveal a lot of arithmetic information about the finite extension of $K_\p$ determined by $F$ \cite{Ok,okutsu}. The initials OM stand indistinctly for Ore-MacLane or Okutsu-Montes. 

The Montes algorithm has been used as the core of several arithmetic routines to compute prime ideal decomposition, integral bases and the discriminant of $L/K$, generators of prime ideals, the $\P$-adic valuation, $v_\P\colon L^*\longrightarrow \Z$, the reduction mapping, $B\longrightarrow B/\P$, the Chinese remainder algorithm in $B$, and the $\p$-valuation of discriminants and resultants of polynomials with coefficients in $K$ \cite{algorithm,bases,newapp,Ndiff}.    

Also, if the Montes algorithm is combined with the single-factor lifting algorithm \cite{GNP}, they yield a fast factorization routine for polynomials over local fields, which turns into an acceleration of some of the above mentioned routines.   

The complexity of the Montes algorithm was analyzed by D. Ford-O. Veres \cite{FV} and S. Pauli \cite{pauli}. Assuming $\p$ small, they obtained an estimation of $O(n^{2+\epsilon}\delta^{2+\epsilon})$ word operations for the algorithm used as an irreducibi\-lity test for polynomials over local fields, where $n=[L\colon K]$ and $\delta=v_\p(\dsc(f))$. Then, by natural extrapolation arguments they concluded that this estimation is valid for the general algorithm too. 

In this paper, we present a new estimation for the complexity of the Montes algorithm. To this end, we find the least precision $\nu$ such that the polynomial $f(x) \md{\p^\nu}$ contains sufficient information to detect that $f(x)$ is irreducible over $\oo_\p$, and the least precision such that a factorization of $f(x) \md{\p^\nu}$ determines a ``sufficiently good" approximate factorization of $f(x)$ over $\oo_\p$.

In section \ref{secOkutsu} we review the role of the Okutsu invariants of the irreducible factors of $f(x)$ over $\oo_\p$, which are essential for our purposes. In section \ref{secOkDisc}, we introduce a new Okutsu invariant, the \emph{exponent of the Okutsu discriminant}, which is a key ingredient to prove that the irreducibility of $f(x)$ over $\oo_\p$ may be tested by working at precision $\nu=\lfloor2\delta/n\rfloor+1$ (Theorem \ref{bound}). In section \ref{secOMfac} we introduce the concept of \emph{OM factorization}, giving a precise sense to what we mean by a ``sufficiently good" approximate factorization. We show that the OM representations satisfying certain properties are the adequate objects to deal with OM factorizations from a computational perspective, and we prove that an OM factorization of $f(x)$ over $\oo_\p$ can be found by working at precision $\nu=\delta+1$ (Theorem \ref{precision}). In section \ref{secAlgo}, we review the Montes algorithm as a device to compute an OM factorization of $f(x)$ over $\oo_\p$. Finally, in section \ref{secComplexity} we use these results to obtain an estimation of $O(n^{2+\epsilon}+\delta^{2+\epsilon})$ word operations for the complexity of the Montes algorithm used as a polynomial irreducibility test, and an estimation of $O(n^{2+\epsilon}+n^{1+\epsilon}\delta^{2+\epsilon})$ word operations for the complexity of the general algorithm.
This estimation yields improved estimations for the complexity of all the arithmetic routines mentioned above. For instance, we deduce an estimation of $O\left(n^{2+\epsilon}+n^{1+\epsilon}\delta^{2+\epsilon}+n^2\nu^{1+\epsilon}\right)$ word operations for the complexity of the factorization of $f(x)$ over $\oo_\p[x]$, with an arbitrary prescribed precision $\nu$ (Theorem \ref{factorization}). The best known previous estimation for the factorization of polynomials over local fields had total degree $4+\epsilon$ in $n$, $\delta$ and $\nu$ \cite{GNP}. 

\section{Okutsu invariants of an irreducible polynomial over a local field}\label{secOkutsu} 
Let $k$ be a local field, i.e. a locally compact and complete field with respect to a discrete valuation $v$. Let $\oo$ be the valuation ring of $k$, $\m$ the maximal ideal, $\pi\in\m$ a generator of $\m$ and $\F=\oo/\m$ the residue field, which is a finite field. Let $p$ be the characteristic of $\F$. 

Let $\ks\subset \kb$ be the separable closure of $k$ inside a fixed algebraic closure. Let $v\colon \kb\to \Q\cup\{\infty\}$, be the canonical extension of the discrete valuation $v$ to $\kb$, normalized by $v(k)=\Z$.

Let $F(x)\in\oo[x]$ be a monic irreducible separable polynomial, $\t\in \ks$ a root of $F(x)$, and $L=k(\t)$ the finite separable extension of $k$ generated by $\t$. Denote $n:=[L\colon k]=\deg F$. Let $\oo_L$ be the ring of integers of $L$, $\m_L$ the maximal ideal and $\F_L$ the residue field. 
We indicate with a bar, $\raise.8ex\hbox{---}\colon \oo[x]\longrightarrow \F[x]$,
the canonical homomorphism of reduction of polynomials modulo $\m$. 

Let $[\phi_1,\dots,\phi_r]$ be an \emph{Okutsu frame} of $F(x)$, and let $\phi_{r+1}$ be an \emph{Okutsu approxi\-mation} to $F(x)$. That is, $\phi_1,\dots,\phi_{r+1}\in\oo[x]$ are monic separable polynomials of strictly increasing degree: $$1\le m_1:=\deg\phi_1<\cdots <m_r:=\deg\phi_r<m_{r+1}:=\deg\phi_{r+1}=n,$$ and for any monic polynomial $g(x)\in\oo[x]$ we have: 
\begin{equation}\label{frame}
m_i\le \deg g<m_{i+1}\ \Longrightarrow\ \dfrac{v(g(\t))}{\deg g}\le\dfrac{v(\phi_i(\t))}{m_i}<\dfrac{v(\phi_{i+1}(\t))}{m_{i+1}},
\end{equation}
for $0\le i\le r$, with the convention that $m_0=1$ and $\phi_0(x)=1$. It is easy to deduce from (\ref{frame}) that the polynomials $\phi_1(x),\dots,\phi_{r+1}(x)$ are all irreducible in $\oo[x]$. 

The length $r$ of the frame is called the \emph{Okutsu depth} of $F(x)$. 
Okutsu frames were introduced by K. Okutsu in \cite{Ok} as a tool to construct integral bases. Okutsu approximations were introduced in \cite{okutsu}, where it is shown that the family $\phi_1,\dots,\phi_r,\phi_{r+1}$ determines an \emph{optimal $F$-complete type of order $r+1$}:
\begin{equation}\label{OM}
\ty_F=\begin{cases}
(\psi_0;(\phi_1,\lambda_1,\psi_1);\cdots;(\phi_r,\lambda_r,\psi_r);(\phi_{r+1},\lambda_{r+1},\psi_{r+1})),
\mbox{ or}\\
(\psi_0;(\phi_1,\lambda_1,\psi_1);\cdots;(\phi_r,\lambda_r,\psi_r);(F,-\infty,\hbox{---})), 
\end{cases}
\end{equation}
according to $\phi_{r+1}\ne F$ or  $\phi_{r+1}=F$, respectively. We call $\ty_F$ an \emph{OM representation} of $F$. In the case $\phi_{r+1}=F$, we say that the OM representation is \emph{exact}.

Any OM representation of the polynomial $F$ carries (stores) several invariants and operators yielding strong arithmetic information about $F$ and the extension $L/k$.
Let us recall some of these invariants and operators.

Attached to the type $\ty_F$, there is a family of discrete valuations of the rational function field $k(x)$, the \emph{MacLane valuations}:
$$
v_i\colon k(x)\longrightarrow \Z\cup\{\infty\},\quad 1\le i\le r+1,
$$
such that $0=v_1(F)<\cdots<v_{r+1}(F)$. The $v_1$-value of a polynomial in $k[x]$ is the minimum of the $v$-values of its coefficients.

Also, $\ty_F$ determines a family of Newton polygon operators:
$$
N_i\colon k[x]\longrightarrow 2^{\R^2},\quad 1\le i\le r+1,
$$where $2^{\R^2}$ is the set of subsets of the Euclidean plane. Any non-zero polynomial $g(x)\in k[x]$ has a canonical $\phi_i$-development:
$$
g(x)=\sum\nolimits_{0\le s}a_s(x)\phi_i(x)^s,\quad \deg a_s<m_i,
$$ and the polygon $N_i(g)$ is the lower convex hull of the set of points $(s,v_i(a_s\phi_i^s))$. Usually, we are only interested in the principal polygon $N_i^-(g)\subset N_i(g)$ formed by the sides of negative slope. For all $1\le i\le r$, the Newton polygons $N_i(F)$ and $N_i(\phi_{i+1})$ are one-sided and they have the same slope, which is a negative rational number $\lambda_i\in\Q_{<0}$.  
The Newton polygon $N_{r+1}(F)$ is one-sided and it has an (extended) integer negative slope, which we denote by $\lambda_{r+1}\in\Z\cup\{-\infty\}$.  

The triple $(\phi_i,v_i,\lambda_i)$ determines the discrete valuation $v_{i+1}$ as follows: for any  non-zero polynomial $g(x)\in K[x]$,
take a line of slope $\lambda_i$ far below $N_i(g)$ and let it shift upwards till it touches the polygon for the first time;
if $u$ is the ordinate of the point of intersection of this line with the vertical axis, then $v_{i+1}(g)=e_i u$.

There is a chain of finite extensions: $\F=\F_0\subset \F_1\subset\cdots\subset\F_{r+1}=\F_L$. The type $\ty_F$ stores monic irreducible polynomials $\psi_i(y)\in\F_i[y]$ such that $\F_{i+1}\simeq \F_i[y]/(\psi_i(y))$. We have $\psi_i(y)\ne y$, for all $i>0$. Finally, for every negative rational number $\lambda$, there are \emph{residual polynomial} operators:
$$
R_{\lambda,i}\colon k[x] \longrightarrow \F_i[y],\quad 0\le i\le r+1.
$$ 
We define $R_i:=R_{\lambda_i,i}$. For all $0\le i\le r$, we have $R_i(F)\sim \psi_i^{\omega_{i+1}}$ and $R_i(\phi_{i+1})\sim\psi_i$, where the symbol $\sim$ indicates that the polynomials coincide up to a multiplicative constant in $\F_i^*$. For $i=0$ we have $R_0(F)=\overline{F}=\psi_0^{\omega_1}$ and $R_0(\phi_1)=\overline{\phi_1}=\psi_0$. The exponents $\omega_{i+1}$ are all positive and $\omega_{r+1}=1$. The operator $R_{r+1}$ is defined only when $\phi_{r+1}\ne F$; in this case, we have also $R_{r+1}(F)\sim \psi_{r+1}$, with $\psi_{r+1}(y)\in\F_{r+1}[y]$ monic of degree one such that $\psi_{r+1}(y)\ne y$. 

From these data some more numerical invariants are deduced. Initially we take:
$$m_0:=1,\quad f_0:=\deg \psi_0,\quad e_0:=1,\quad h_0:=V_0:=\mu_0:=\nu_0=0.
$$Then, we define for all $1\le i\le r+1$:
$$
\as{1.2}
\begin{array}{l}
h_i,\,e_i \ \mbox{ positive coprime integers such that }\la_i=-h_i/e_i,\\
f_i:=\deg \psi_i,\\
m_i:=\deg \phi_i=e_{i-1}f_{i-1}m_{i-1}=(e_0\,e_1\cdots e_{i-1})(f_0f_1\cdots f_{i-1}),\\
\mu_i:=\sum_{1\le j\le i}(e_jf_j\cdots e_if_i-1)h_j/(e_1\cdots e_j),\\
\nu_i:=\sum_{1\le j\le i}h_j/(e_1\cdots e_j),\\
V_i:=v_i(\phi_i)=e_{i-1}f_{i-1}(e_{i-1}V_{i-1}+h_{i-1})=(e_0\cdots e_{i-1})(\mu_{i-1}+\nu_{i-1}).
\end{array}
$$
\as{1.}

The general definition of a type may be found in \cite[Sec. 2.1]{HN}. In later sections, we shall consider types which are not necessarily optimal nor $F$-complete. So, it may be convenient to distinguish these two properties among all features of a type that we have just mentioned.

\begin{definition}\label{optimal}Let  $\ty=(\psi_0;(\phi_1,\lambda_1,\psi_1);\cdots;(\phi_i,\lambda_i,\psi_i))$ be a type of order $i$ and denote $m_{i+1}:=e_if_im_i$. Let $g(x)\in K[x]$ be a polynomial.\medskip

\noindent{$\bullet$} We say that $\ty$ is \emph{optimal} if $m_1<\cdots<m_i$.
We say that $\ty$ is \emph{strongly optimal} if $m_1<\cdots<m_i<m_{i+1}$.
\medskip

\noindent{$\bullet$} We define $\ord_\ty(g):=\ord_{\psi_{i}}R_{i}(g)$ in $\F_i[y]$.
If $\ord_\ty(g)>0$, we say that $\ty$ \emph{divides} $g(x)$, and we write $\ty\mid g(x)$. This function $\ord_\ty$ behaves well with respect to products: $\ord_\ty(gh)=\ord_\ty(g)+\ord_\ty(h)$.  
\medskip

\noindent{$\bullet$}We say that $\ty$ is \emph{$g$-complete} if $\ord_\ty(g)=1$.
\medskip

\noindent{$\bullet$} A \emph{representative} of $\ty$ is a monic polynomial $\phi(x)\in\oo[x]$ of degree $m_{i+1}$, such that $\ord_\ty(\phi)=1$. This polynomial is necessarily irreducible in $\oo[x]$. The degree $m_{i+1}$ is minimal among all polynomials satisfying this condition.
\medskip

\noindent{$\bullet$} For any $0\le j\le i$, the \emph{truncation} of $\ty$ at level $j$, $\op{Trunc}_j(\ty)$, is the type of order $j$ obtained from $\ty$ by dropping all levels  higher than $j$. We have $\ord_{\op{Trunc}_j(\ty)}(g)\ge (e_{j+1}f_{j+1})\cdots (e_if_i)\ord_\ty(g)$. 
\end{definition}
 
Thus, for a general type of order $i$ dividing $F$, we have $m_1\mid\cdots\mid m_i$ and $\omega_i>0$, but not necessarily $m_1<\cdots<m_i=\deg F$, and $\omega_i=1$. These were particular properties of our optimal and $F$-complete type $\ty_F$ of order $i=r+1$, constructed from an Okutsu frame and an Okutsu approximation to $F$. 

An irreducible polynomial $F$ admits infinitely many different OM representations. However, the numeri\-cal invariants $e_i,f_i,h_i$, for $0\le i\le r$, and the MacLane valuations 
$v_1,\dots,v_{r+1}$ attached to $\ty_F$, are canonical invariants of $F$. 

The data $\lambda_{r+1},\psi_{r+1}$ are not invariants of $F$; they depend on the choice of the Okutsu approximation $\phi_{r+1}$. The integer slope $\lambda_{r+1}=-h_{r+1}$ measures how close is $\phi_{r+1}$ to $F$. We have $\phi_{r+1}=F$ if and only if $h_{r+1}=\infty$.
 
\begin{definition}
An \emph{Okutsu invariant} of $F(x)$ is a rational number that depends only on $e_1,\dots,e_r,f_0,f_1,\dots,f_r,h_1,\dots,h_r$.
\end{definition}

We are specially interested in the following invariants of the polynomial $F(x)$:
$$
\begin{array}{l}
e(F):=e(L/k),  \mbox{ the ramification index of }L/k,\\
f(F):=f(L/k), \mbox{ the residual degree of }L/k,\\
\mu(F):=\max\{v(g(\t))\mid g(x)\in\oo[x] \ \mbox{monic of degree less than }n\},\\
\delta(F):=v(\dsc(F)).
\end{array}
$$

The different ideal of $L/k$ is $\df(L/k)=(\m_L)^{e-1+\rho}$, for some integer $\rho\ge0$, which is not an Okutsu invariant. Also, $\rho=0$ if and only if $L/k$ is tamely ramified. 

\begin{proposition}{\cite[Cor. 3.8]{HN}, \cite[Cor. 1.8]{Ndiff}}\label{okutsu}
$$\as{1.2}
\begin{array}{l}
e(F)=e_0\,e_1\cdots e_r,\quad f(F)=f_0f_1\cdots f_r,\\
\mu(F)=\mu_r=\sum_{1\le j\le r}(e_jf_j\cdots e_rf_r-1)h_j/(e_1\cdots e_j),\\
\delta(F)=n\mu(F)+f(F)\rho.
\end{array}
$$
\end{proposition}

Thus, $e(F),\,f(F)$ and $\mu(F)$ are Okutsu invariants of $F$, but $\delta(F)$ is not. Nevertheless, the lower bound by an Okutsu invariant, $\delta(F)\ge n\mu(F)$, will be essential for our purposes.

\begin{definition}\label{length} 
The \emph{length} of a Newton polygon $N$ is the abscissa of its right end point; we denote it by $\ell(N)$. 
\end{definition}

The following lemma will be frequently used.

\begin{lemma}\cite[Prop. 2.7,Lem. 2.17,Thm. 3.1]{HN}\label{previous}
Let $\ty$ be a type of order $r$.
\begin{enumerate}
\item $v_i(a)=e_0\cdots e_{i-1}v(a)$, for all $a\in k$ and all $1\le i\le r+1$.
\item $\ell(N_{r+1}(g))=\ord_\ty(g)$, for any non-zero polynomial $g(x)\in k[x]$.
\item $v(\phi_i(\t))=(V_i+|\lambda_i|)/(e_0\cdots e_{i-1})=\mu_{i-1}+\nu_i$, for all $1\le i\le r+1$.
\item $v(\phi_i(\t))/m_i=V_{i+1}/(m_{i+1}e_0\cdots e_{i})$, for all $1\le i\le r$.
\end{enumerate}
\end{lemma}

We end this background section by recalling the \emph{Okutsu equivalence} of irreducible separable polynomials over $\oo$, and the concept of \emph{width} of such a polynomial.

\begin{lemma}{\cite[Lem. 3.1]{GNP}}\label{representative}
Let  $\ty$ be a strongly optimal type of order $r$, and let $\phi\in\oo[x]$ be a monic polynomial of degree $m_{r+1}$. 
Let $F\in\oo[x]$ be an irreducible separable polynomial such that $\ty\mid F$, and let $\t\in\ks$ be a root of $F$. Then, the following conditions are equivalent:
\begin{enumerate}
\item[(a)] $\phi$ is a representative of $\ty$.
\item[(b)] $v(\phi(\t))>V_{r+1}/(e_0\cdots e_{r})=(m_{r+1}/m_r)v(\phi_r(\t))$.
\end{enumerate}
\end{lemma}
 
\begin{definition}\label{okequiv}
Let $F\in\oo[x]$ be a monic irreducible separable polynomial of Okutsu depth $r$, and let $\ty_F$ be an OM representation of $F$ as in (\ref{OM}). Let $\ty:=\op{Trunc}_r(\ty_F)$. We say that a monic polynomial $G\in\oo[x]$ is an \emph{Okutsu approximation to $F$}, and we write $F\approx G$,  if $G$ is a representative of $\ty$. 

We also say that $F$ and $G$ are \emph{Okutsu equivalent} polynomials.
\end{definition}

By Lemma \ref{representative}, this definition does not depend on the choice of the OM representation of $F$. The binary relation $\approx$ is an equivalence relation on the set of all monic irreducible separable polynomials in $\oo[x]$ \cite[Lem. 4.3]{okutsu}. Okutsu equivalent polynomials have the same Okutsu invariants and the same MacLane valuations \cite[Cor. 3.7]{okutsu}.

For $F$ as above, and $1\le i\le r+1$, let $\op{Rep}_i(F)\subseteq \oo[x]$ be the set of  all representatives of $\op{Trunc}_{i-1}(\ty_F)$. Consider:  
$$\vv_i:=\left\{v(\phi(\t))\mid \phi\in \op{Rep}_i(F)\right\}\subseteq \Q\cup \{\infty\}.$$ 
By the formula (\ref{frame}), $\phi_i\in\op{Rep}_i(F)$ and $v(\phi_i(\t))=\op{Max}(\vv_i)$, for all $1\le i\le r$. By definition, $\op{Rep}_{r+1}(F)$ is the set of all Okutsu approximations to $F(x)$. The set $\vv_{r+1}$ is not finite, and it contains $ \infty$, because $F\in\op{Rep}_{r+1}(F)$.

The sets $\vv_1,\dots,\vv_r$ are finite and easy to describe \cite[Prop. 3.4]{GNP}.

\begin{proposition}\label{compwidth}
For any $\lambda\in\Q$, let $M_\lambda:=\{m\in \Z\mid 1\le m<|\lambda|\} \cup \{|\lambda|\}$. Then, 
$$
\vv_i=\left\{(V_i+m)/(e_0\cdots e_{i-1})\mid m\in M_{\lambda_i}\right\},$$
for all $1\le i\le r$. 
In particular, $\#\vv_i=\lceil|\lambda_i|\rceil=\lceil h_i/e_i\rceil$.
\end{proposition}

The \emph{width} of $F(x)$ is defined to be the vector of positive integers,
$$(\#\vv_1,\dots,\#\vv_r)=(\lceil h_1/e_1\rceil,\cdots,\lceil h_r/e_r\rceil).$$ As we shall see in section \ref{secComplexity}, it is a fundamental invariant for the analysis of the complexity of the Montes algorithm.

\section{The Okutsu discriminant}\label{secOkDisc}
We keep all notation from the previous section.
In this section we introduce a new Okutsu invariant of an irreducible polynomial $F(x)\in\oo[x]$, linked to the problem of determining the least exponent $\nu$ such that all polynomials of degree $n=\deg F$, belonging to $F(x)+\m^\nu[x]$, are irreducible in $\oo[x]$.

\begin{definition}\label{okutsudisc}
Let $F(x)\in\oo[x]$ be a monic irreducible separable polynomial of degree $n$ and $\ty_F$ an OM representation of $F$ as in (\ref{OM}). If $r$ is the Okutsu depth of $F(x)$, we define the \emph{Okutsu discriminant} of $F(x)$ as the ideal $\m^{\delta_0(F)}$, where
\begin{equation}\label{delta0}
\delta_0(F):=\dfrac{V_{r+1}}{e(F)}=\mu_r+\nu_r=\sum\nolimits_{1\le i\le r}\dfrac{|\lambda_i|}{e_0\cdots e_{i-1}}\,\dfrac{n}{m_i}. 
\end{equation}
\end{definition}

The exponent $\delta_0(F)$ of the Okutsu discriminant coincides, up to a certain normalization, 
with the ordinate of the left end point of $N_r(F)$. 

\begin{lemma}\label{delta0props}
With the above notation, denote $u_i:=v_i(a_{0,i}(F))$, for $1\le i\le r$, where $a_{0,i}(F)\in\oo[x]$ is the $0$-th coefficient of the $\phi_i$-development of $F$. Then,
\begin{enumerate}
\item $u_1<u_2/e_1<\cdots<u_r/(e_0\cdots e_{r-1})=\delta_0(F)$.
\item $\delta_0(F)\le 2\delta(F)/n$, and equality holds if and only if either $r=0$, or $r=1$, $e_1f_1=2$, $p>e_1$.
\end{enumerate}
\end{lemma}

\begin{proof}Denote $\omega_i=n/m_i=(e_if_i)\cdots (e_rf_r)$. The Newton polygon $N_i(F)$ is one-sided, with end points $(0,u_i)$ and $(\omega_i,v_i(F))$ \cite[Lem. 2.17]{HN}. Also, the leading term of the $\phi_i$-adic expansion of $F$ is $\phi_i^{\omega_i}$. Thus, $v_i(F)=\omega_iV_i$ and  
\begin{equation}\label{ui}
\dfrac{u_i}{e_0\cdots e_{i-1}}=\dfrac{v_i(F)+\omega_i|\lambda_i|}{e_0\cdots e_{i-1}}=\dfrac{\omega_i(V_i+|\lambda_i|)}{e_0\cdots e_{i-1}}=n\dfrac{v(\phi_i(\t))}{m_i},
\end{equation}
the last equality by Lemma \ref{previous},(3).
By the properties (\ref{frame}) of the Okutsu frame, $u_1<u_2/e_1<\cdots<u_r/(e_0\cdots e_{r-1})$. Also,  by Lemma \ref{previous},(4),
$$u_r/(e_0\cdots e_{r-1})=(n/m_r)v(\phi_r(\t))=V_{r+1}/e(F)=\delta_0(F).
$$

On the other hand, since $e_if_i>1$, for all $1\le i\le r$, we have $\nu_r\le \mu_r=\mu(F)$. Thus, $\delta_0(F)\le 2\mu(F)\le 2\delta(F)/n$, by Proposition \ref{okutsu}. Also, equality holds if and only if $\mu_r=\nu_r$ and $F$ determines a tamely ramified extension of $k$ (i.e. $\rho=0$). The formulas for $\mu_r$, $\nu_r$ in section \ref{secOkutsu} lead to the conditions of item (2).   
\end{proof}

The aim of this section is to prove the following result.

\begin{theorem}\label{bound}
Let $F(x),G(x)\in\oo[x]$ be monic separable polynomials of degree $n$, such that $F\equiv G\md{\m^\nu}$, for some positive exponent $\nu$.

\begin{enumerate}
\item If $F$ is irreducible and $\nu>\delta_0(F)$, then $G$ is irreducible and $G\approx F$.
\item If $G$ is irreducible and $\nu>2\delta(F)/n$, then $F$ is irreducible and  $F\approx G$.
\end{enumerate}
\end{theorem}

\begin{corollary}\label{testprecision}
Let $F(x),G(x)\in\oo[x]$ be monic separable polynomials of degree $n$, such that $F\equiv G\md{\m^\nu}$, for $\nu>2\delta(F)/n$. Then, $F$ is irreducible if and only if $G$ is irreducible. If this is the case, the extensions of $k$ determined by $F$ and $G$ have isomorphic maximal tamely ramified subextensions. 
\end{corollary}

\begin{proof}
By Theorem \ref{bound} and Lemma \ref{delta0props}, $F$ is irreducible if and only if $G$ is irreducible, and in this case $F\approx G$. Also, if  $F\approx G$, then
for adequate choices of roots $\t,\t'\in\ks$ of $F$ and $G$, respectively, the fields $k(\t)$ and $k(\t')$ have the same maximal tamely ramified subextension \cite{Ok}, \cite[Prop. 2.7]{okutsu}. 
\end{proof}

The first item of Theorem \ref{bound} follows immediately from Lemma \ref{delta0props}. In fact, for $\t\in\ks$ a root of $F$, the assumptions of the first item imply that $v(G(\t))>\delta_0(F)=nv(\phi_r(\t))/m_r$, and this is precisely the condition to be an Okutsu approximation to $F$ (cf. Lemma \ref{representative}). As mentioned in Definition \ref{optimal}, this implies that $G$ is irreducible.

The second item is more subtle and its proof more involved. We need some previous results.

\begin{definition}\label{type}
Let $\ty=(\psi_0;(\phi_1,\lambda_1,\psi_1);\cdots;(\phi_{i-1},\lambda_{i-1},\psi_{i-1}))$ be a type of order $i-1\ge 0$, and let $F(x)\in\oo[x]$ be a monic polynomial. We say that $F(x)$ is \emph{a polynomial of type $\ty$} if it satisfies the following conditions:
\begin{enumerate}
\item $R_0(F)=\overline{F}=\psi_0^{a_0}$, for a certain positive exponent $a_{0}$,
\item $N_j(F)$ is one-sided of slope $\lambda_j$, for all $1\le j<i$,
\item $R_j(F)\sim \psi_j^{a_{j}}$, for a certain positive exponent $a_{j}$, for all $1\le j<i$. 
\end{enumerate}
\end{definition}

For instance, if $F$ is irreducible and $\ty_F$ is an OM representation of $F$, then $F$ is of type $\ty_F$. The following properties of the polynomials of a certain type are taken from  \cite[Lem. 2.4, Cor. 2.18]{HN}.

\begin{lemma}\label{oftype}
Let $\ty$ be a type of order $i-1\ge0$, and let $F(x)\in\oo[x]$ be a monic polynomial of positive degree. Then, the following conditions are equivalent:
\begin{enumerate}
\item $F$ is of type $\ty$.  
\item $\deg F=m_i\ord_\ty(F)$.
\item All irreducible factors of $F$ in $\oo[x]$ are divisible by $\ty$.
\end{enumerate}
In this case, we have $N_i(F)=N_i^-(F)$.  
\end{lemma}

\begin{lemma}\label{resultant}
Let $\ty$ be as above and let $F,G\in\oo[x]$ be monic irreducible separable polynomials, both divisible by $\ty$. 
Let $\ell(F),\,\ell(G),\,\lambda(F),\,\lambda(G)$ be the lengths and the slopes of the Newton polygons $N_i(F)$, $N_i(G)$, respectively.
Then, 
$$
v(\res(F,G))\ge f_0\cdots f_{i-1}\ell(F)\ell(G)\left(V_i+\min\{|\lambda(F)|,|\lambda(G)|\}\right).
$$
\end{lemma}

\begin{proof}
For all $0\le j< i$, denote $\ell_{j+1}(F):=\ell(N_{j+1}(F))=\ord_{\op{Trunc}_j(\ty)}(F)=\ord_{\psi_{j}}R_{j}(F)$, the last equalities by Lemma \ref{previous},(2). Since $R_{j}(F)\sim\psi_j^{\ell_{j+1}(F)}$ and $\deg R_{j}(F)$ coincides with the degree $\ell_j(F)/e_j$ of the unique side of $N_{j}(F)$, we have  
\begin{equation}\label{rec}
\ell_j(F)=e_{j}\deg R_{j}(F)=e_{j}f_{j}\ell_{j+1}(F)=(e_jf_j)\cdots (e_{i-1}f_{i-1})\ell(F),\quad 1\le j<i.
\end{equation}
We consider an analogous notation and equality for the polynomial $G$.  

We now apply an inequality concerning the $v$-value of the resultant of two polynomials in terms of their Newton polygons \cite[Thm. 4.10]{HN}:
$$
\begin{aligned}
v(\res&(F,G))\ge \res_1(F,G)+\cdots+\res_i(F,G):=\\
&\ \sum\nolimits_{1\le j< i}f_0\cdots f_{j-1}\ell_{j}(F)\ell_{j}(G)|\lambda_j|+f_0\cdots f_{i-1}\ell(F)\ell(G)\min\{|\lambda(F)|,|\lambda(G)|\}\\
&=\ f_0\cdots f_{i-1}\ell(F)\ell(G)\left(V_i+\min\{|\lambda(F)|,|\lambda(G)|\}\right).
\end{aligned}
$$
the last equality by (\ref{rec}) and the explicit formula for $V_i$ in section \ref{secOkutsu}.
\end{proof}

\begin{lemma}\label{technical}
Let $\ty=(\psi_0;(\phi_1,\lambda_1,\psi_1);\cdots;(\phi_{i-1},\lambda_{i-1},\psi_{i-1}))$ be a strongly optimal type of order $i-1\ge 0$, and $\phi(x)\in \oo[x]$ a representative of $\ty$. 
Let $F(x)\in\oo[x]$ be a monic polynomial of type $\ty$ and degree $n>m_i$. Then,
$$
\dfrac{v_i(F)+\ell|\lmn|}{e_0\cdots e_{i-1}}\le \dfrac{2\delta(F)}{n},
$$
where $\delta(F):=v(\dsc(F))$, $\ell$ is the length of the Newton polygon $N_i(F)$ with respect to the pair $(\ty,\phi)$, and $\lmn$ is the slope of $N_i(F)$ for which $|\lmn|$ is minimal.
\end{lemma}

\begin{proof}
Let $F=F_1\cdots F_g$ be the factorization of $F$ into a product of monic irreducible polynomials in $\oo[x]$, with degrees $n_1,\dots,n_g$, respectively. By Lemma \ref{oftype}, all factors $F_s(x)$ are of type $\ty$, $N_i(F)=N_i^-(F)$, and $N_i(F_s)=N_i^-(F_s)$. 

For $1\le s \le g$ and $1\le j\le i$, we introduce the following notation (see Figure \ref{figSide}):\medskip

$\ell:=\ell(N_i(F))$, \quad$\ell_{j,s}:=\ell(N_j(F_s))$, \quad$\ell_s:=\ell_{i,s}=\ell(N_i(F_s))$, 

$u_{i,s}:=$ the ordinate of the left end point of $N_i(F_s)$,

$\mu_s:=$ the slope of $N_i(F_s)$.  
\medskip

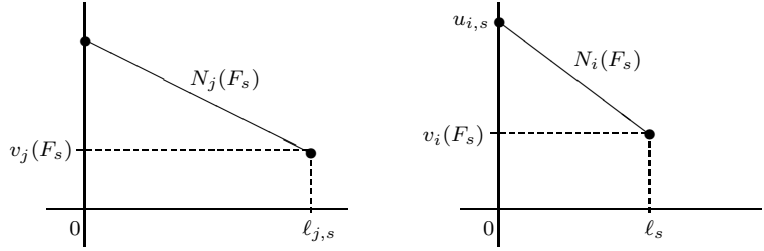
\begin{figure}\caption{Newton polygons $N_j(F_s)$, $N_i(F_s)$, for $1\le j<i$.}\label{figSide}
\setlength{\unitlength}{5mm}
\begin{center}
\begin{picture}(20,5.7)
\put(.85,4.5){$\bullet$}\put(6.85,1.5){$\bullet$}
\put(0,0.2){\line(1,0){8}}
\put(1,4.7){\line(2,-1){6}}
\put(1,-.8){\line(0,1){6.5}}
\put(.6,-.5){\begin{footnotesize}$0$\end{footnotesize}}
\put(6.8,-.5){\begin{footnotesize}$\ell_{j,s}$\end{footnotesize}}
\put(-1,1.6){\begin{footnotesize}$v_j(F_s)$\end{footnotesize}}
\put(3.8,3.5){\begin{footnotesize}$N_j(F_s)$\end{footnotesize}}
\multiput(7,.1)(0,.25){7}{\vrule height2pt}
\multiput(.8,1.75)(.25,0){25}{\hbox to 2pt{\hrulefill }}
\put(11.85,5){$\bullet$}\put(15.85,2){$\bullet$}
\put(11,0.2){\line(1,0){8}}
\put(12,5.2){\line(4,-3){4}}
\put(12,-.8){\line(0,1){6.5}}
\put(11.6,-.5){\begin{footnotesize}$0$\end{footnotesize}}
\put(15.9,-.5){\begin{footnotesize}$\ell_s$\end{footnotesize}}
\put(10,2){\begin{footnotesize}$v_i(F_s)$\end{footnotesize}}
\put(14,4){\begin{footnotesize}$N_i(F_s)$\end{footnotesize}}
\put(10.8,5.1){\begin{footnotesize}$u_{i,s}$\end{footnotesize}}
\multiput(16,.1)(0,.25){9}{\vrule height2pt}
\multiput(11.8,2.2)(.25,0){17}{\hbox to 2pt{\hrulefill }}
\end{picture}
\end{center}
\end{figure}
 
We may have $F_s(x)=\phi(x)$ for some factors. In this case, $N_i(F_s)$ is one-sided of slope $\mu_s=-\infty$ \cite[Sec. 1.1]{HN}, and $u_{i,s}=\infty$, $\ell_s=1$. 

By Lemmas \ref{previous} and \ref{oftype}, we have $n=m_i\ell$ and $n_s=m_i\ell_s$, for all $1\le s\le g$.
By the theorem of the product \cite[Thm. 2.26]{HN}, 
\begin{equation}\label{product}
N_i(F)=N_i(F_1)+\cdots+N_i(F_g),
\end{equation}
so that $\ell=\ell_1+\cdots+\ell_g$ and $|\lmn|=\min_{1\le s \le g}\{|\mu_s|\}$.
Now, we divide the factors $F_s$ into two categories, according to $\ell_s>1$ or $\ell_s=1$. 

If $\ell_s>1$, then $\deg\phi=m_i<n_s$. Let $\t_s\in\overline{k}$ be a root of $F_s$ and choose a representative $\phi_i$ of $\ty$ such that the value $v(\phi_i(\t_s))$ is maximal (cf. Proposition \ref{compwidth}). Denote by $N_i'$ the Newton polygon operator with respect to the pair $(\ty,\phi_i)$; let $\lambda_{i,s}$ be the slope of the one-sided polygon $N'_i(F_s)$, and let $\psi_{i,s}$ be the irreducible factor of the corresponding residual polynomial $R'_i(F_s)$. By \cite[Thm. 3.9]{okutsu}, the Okutsu depth of $F_s$ is greater than or equal to $i$, and the type 
$$
(\psi_0;(\phi_1,\lambda_1,\psi_1);\cdots;(\phi_{i-1},\lambda_{i-1},\psi_{i-1});(\phi_i,\lambda_{i,s},\psi_{i,s})), 
$$
is the truncation of an OM representation (\ref{OM}) of $F_s$. On the other hand, \cite[Thm. 3.1]{algorithm} shows that the Newton polygons $N_i(F_s)$, $N_i'(F_s)$ have the same right end point, and $|\mu_s|\le|\lambda_{i,s}|$. Thus, $u_{i,s}$ is less than or equal to the ordinate of the left end point of $N'_i(F_s)$, and Lemma \ref{delta0props} and (\ref{ui}) show that:
\begin{equation}\label{greater1}
 2\delta(F_s)\ge \dfrac{n_su_{i,s}}{e_0\cdots e_{i-1}}=\dfrac{n_s\ell_s(V_i+|\mu_s|)}{e_0\cdots e_{i-1}}\ge\dfrac{(\ell_s)^2m_i(V_i+|\lmn|)}{e_0\cdots e_{i-1}}.
\end{equation}

On the other hand, if $\ell_s=1$, the type $\ty$ is $F_s$-complete (cf. Definition \ref{optimal}), $\deg F_s=m_i$ and the Okutsu depth of $F_s$ is $i-1$. In this case,  the ordinate $u_{i,s}$ is not a canonical invariant of $F_s$; for instance, we may have $u_{i,s}=\infty$, if $F_s=\phi$. Nevertheless, if $i>1$, let us denote by $u_{i-1,s}$ the ordinate of the left end point of $N_{i-1}(F_s)$; by the very definition of the MacLane valuation $v_i$, we have $\ell_sV_i=v_i(F_s)=e_{i-1}u_{i-1,s}$, and Lemma \ref{delta0props} shows that: 
\begin{equation}\label{equal1}
 2\delta(F_s)\ge\dfrac{n_su_{i-1,s}}{e_0\cdots e_{i-2}}=\dfrac{n_s\ell_sV_i}{e_0\cdots e_{i-1}}=\dfrac{(\ell_s)^2m_iV_i}{e_0\cdots e_{i-1}}.
\end{equation}
If $i=1$, we have $V_1=0$, so that (\ref{equal1}) holds in this case too.  

We are ready to prove the lemma. On one hand, since $v_i(F)=\ell V_i$, we have
$$
n(v_i(F)+\ell|\lmn|)=n\ell(V_i+|\lmn|)=m_i\ell^2(V_i+|\lmn|).
$$On the other hand, since $f_0\cdots f_{i-1}=m_i/(e_0\cdots e_{i-1})$ and 
$$
\delta(F)=\sum\nolimits_{1\le s\le g}\delta(F_s)+2\sum\nolimits_{1\le s<t\le g}v(\res(F_s,F_t)),
$$
by (\ref{greater1}), (\ref{equal1}) and Lemma \ref{resultant}, we get:
\begin{align*}
2e_0\cdots e_{i-1}\delta(F)\ge&\ m_iV_i\left(\sum\nolimits_{1\le s\le g}(\ell_s)^2+4\sum\nolimits_{1\le s<t\le g}\ell_s\ell_t\right)+\\+&\ m_i|\lmn|\left(\sum\nolimits_{s\in I}(\ell_s)^2+4\sum\nolimits_{1\le s<t\le g}\ell_s\ell_t\right),
\end{align*}
where $I:=\{1\le s\le g\mid \ell_s>1\}$. Thus, in order to prove the lemma
it is sufficient to check that:
$$
\sum\nolimits_{s\in I}(\ell_s)^2+4\sum\nolimits_{1\le s<t\le g}\ell_s\ell_t\ge (\ell_1+\cdots+\ell_g)^2.
$$
It is an easy exercise to show that this is always the case, with the only exception $g=1$, $\ell_1=1$. But in this case, $\deg F=m_i$, against our assumption. 
\end{proof}

\begin{lemma}\label{belowline}
Let $\ty$ be a type of order $i-1$ and $\phi$ a representative of $\ty$. Let $F,G\in\oo[x]$ be two polynomials such that $F\equiv G\md{\m^\nu}$, for some positive integer $\nu$. Let $S$ be a side of $N_i^-(F)$ of slope $\lambda$ and right end point $(\ell,u)$, such that $u+\ell|\lambda|< e_0\cdots e_{i-1}\nu$. Then, $S$ is a side of $N_i^-(G)$ and $R_{\lambda,i}(F)=R_{\lambda,i}(G)$.
\end{lemma}

\begin{proof}
Let $F(x)=\sum_{0\le s}a_s(x)\phi(x)^s$,  $G(x)=\sum_{0\le s}b_s(x)\phi(x)^s$, be the canonical $\phi$-expansions 
of $F$ and $G$, respectively. For the elements $a\in \oo$, Lemma \ref{previous} shows that $v_i(a)=e_0\cdots e_{i-1}v(a)$; thus, $v_i(F-G)\ge e_0\cdots e_{i-1}\nu$, by hypothesis. Since $F(x)-G(x)=\sum_{0\le s}(a_s(x)-b_s(x))\phi(x)^s$ is the canonical $\phi$-expansion of $F-G$, \cite[Lem. 2.17]{HN} shows that 
$$e_0\cdots e_{i-1}\nu\le v_i(F-G)=\min\{v_i((a_s-b_s)\phi^s)\mid 0\le s\}.
$$Therefore, the two clouds of points $\{(s,v_i(a_s\phi^s))\mid 0\le s\}$, $\{(s,v_i(b_s\phi^s))\mid 0\le s\}$,  
have the same points with ordinate less than $e_0\cdots e_{i-1}\nu$. Let $L$ be the line of slope $\lambda$ containing $S$. No point of the cloud of $F$ lies below the line $L$, and only the points of $S$ lie on this line. The condition  $u+\ell|\lambda|< e_0\cdots e_{i-1}\nu$ implies that the cloud of points of $G$ has the same properties. Thus, $S$ is also a side of $N_i^-(G)$.   

Let $\lambda=-h/e$, with $h,e$ positive coprime integres. Let $v_{i+1}$ be the MacLane valuation determined by  $\ty,\,\phi,\,\lambda$. By the definition of $v_{i+1}$ (cf. section \ref{secOkutsu}),
$$
v_{i+1}(F-G)\ge e_0\cdots e_{i-1}e\nu>e(u+\ell|\lambda|)=v_{i+1}(F)=v_{i+1}(G).
$$ 
Therefore,  $R_{\lambda,i}(F)=R_{\lambda,i}(G)$, by \cite[Prop. 2.8]{HN}. 
\end{proof}

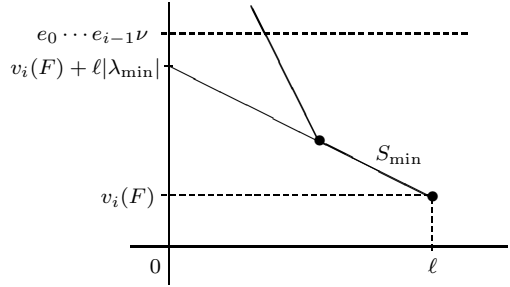
\begin{figure}\caption{Newton polygon $N_i(F)$ in the context of Lemmas \ref{technical}, \ref{belowline}.}\label{figCompare}
\setlength{\unitlength}{5mm}
\begin{center}
\begin{picture}(10,7)
\put(4.85,2.85){$\bullet$}\put(7.85,1.35){$\bullet$}
\put(0,.2){\line(1,0){10}}
\put(8,1.5){\line(-2,1){7}}\put(8,1.52){\line(-2,1){3}}
\put(5,3){\line(-1,2){1.8}}\put(5,3.02){\line(-1,2){1.8}}
\put(.9,5){\line(1,0){.2}}
\put(1,-.8){\line(0,1){7.5}}
\put(6.5,2.5){\begin{footnotesize}$\smn$\end{footnotesize}}
\put(.5,-.5){\begin{footnotesize}$0$\end{footnotesize}}
\put(7.9,-.5){\begin{footnotesize}$\ell$\end{footnotesize}}
\put(-.8,1.35){\begin{footnotesize}$v_i(F)$\end{footnotesize}}
\put(-3.2,4.8){\begin{footnotesize}$v_i(F)+\ell|\lmn|$\end{footnotesize}}
\put(-2.5,5.7){\begin{footnotesize}$e_0\cdots e_{i-1}\nu$\end{footnotesize}}
\multiput(8,.1)(0,.25){6}{\vrule height2pt}
\multiput(.8,1.55)(.25,0){29}{\hbox to 2pt{\hrulefill }}
\multiput(.8,5.85)(.25,0){33}{\hbox to 2pt{\hrulefill }}
\end{picture}
\end{center}
\end{figure}

\noindent{\sl Proof of Theorem \ref{bound}.}
The first item of Theorem \ref{bound} was proved right after Corollary \ref{testprecision}. Let us prove the second item.
Let $r$ be the Okutsu depth of $G(x)$. Consider an OM representation $\ty_G$ of $G(x)$ as in (\ref{OM}), and let
$$\ty:=\op{Trunc}_r(\ty_G)=(\psi_0;(\phi_1,\lambda_1,\psi_1);\cdots;(\phi_r,\lambda_r,\psi_r)). 
$$
Note that $\ty$ is a strongly optimal type that admits $G$ as a representative.

In order to prove the theorem, it is sufficient to show that
\begin{equation}\label{aim}
N_i(F)=N_i(G),\quad R_i(F)=R_i(G),\quad1\le i\le r.
\end{equation}
In fact, $R_r(F)=R_r(G)$ implies that $\ty$ is $F$-complete too; thus, $F$ is a representative of $\ty$, and $F\approx G$, by the definition of $\approx$. 

By hypothesis, $F\equiv G\equiv \psi_0^{a_0}\md{\m}$, for a certain positive exponent $a_0$. Let us prove (\ref{aim}) by induction on $i$. We assume that it is true for all $1\le j<i$ (thus, we make an empty assumption if $i=1$). Since $G$ is a polynomial of type $\ty$, our assumption implies that $F$ satisfies the conditions of Lemma \ref{technical}; thus, 
\begin{equation}\label{lemma}
\dfrac{v_i(F)+\ell|\lmn|}{e_0\cdots e_{i-1}}\le \dfrac{2\delta(F)}{n}<\nu,
\end{equation}
where $\ell=\ell(N_i(F))$ and $\lmn$ is the largest slope of this polygon ($|\lmn|$ is minimal). 

Let $\smn$ be the side of $N_i(F)$ of slope $\lmn$. By Lemma \ref{belowline}, $\smn$ is one of the sides of $N_i(G)$  (see Figure \ref{figCompare}). Since $G$ is irreducible, $N_i(G)$ is one-sided, so that $N_i(G)=\smn$. Thus, the left end point of $\smn$ has abscissa zero, so that $N_i(F)=\smn=N_i(G)$. Also, $R_i(F)=R_i(G)$, again by Lemma \ref{belowline}.  \qed

\begin{remark}\label{reduced}\rm
In \cite{fpr}, the \emph{reduced discriminant} $\m^{\delta^*(F)}$ of an arbitrary polynomial $F(x)\in\oo[x]$ is introduced, and it is shown that Corollary \ref{testprecision} holds with $2\delta^*(F)$ in the place of $2\delta(F)/n$. However, the reduced discriminant does not satisfy $\delta^*(F)\le \delta(F)/n$, so that Theorem \ref{bound} cannot be deduced from this result.
\end{remark}

For instance, suppose $p$ odd and consider $F(x)=x^4+a\pi x^2+b\pi^2\in\oo[x]$, with $ab(a^2-4b)\not\in\m$. This polynomial is irreducible; in fact, if we choose $\phi_1(x)=x$ as a lift of the irreducible factor of $\overline{F}$, the Newton polygon $N_1(F)$ is one-sided of slope $-1/2$ and $R_{-1/2,1}(F)(y)=y^2+\overline{a}y+\overline{b}$ is irreducible in $\F[y]$. One checks easily that 
$$\delta_0(F)=2, \quad \delta^*(F)=3, \quad \delta(F)=6. 
$$
By the first item of Theorem \ref{bound}, any monic polynomial $G(x)\in\oo[x]$ of degree four such that $F\equiv G\md{\m^3}$, is irreducible. If we did not know the irreducibility of $F$, then Corollary \ref{testprecision} shows that we can test its irreducibility by working modulo $\m^4$. However, according to the criterion of the reduced discriminant, we should work modulo $\m^7$ to test the irreducibility of $F$.   

\section{OM factorizations of polynomials}\label{secOMfac}
In this section, we deal with the problem of finding ``sufficiently good" approximations to the irreducible factors of a polynomial in $\oo[x]$.
We first extend the notion of Okutsu equivalence in section \ref{secOkutsu} to non-irreducible polynomials. 

\begin{definition}\label{okequiv2}
Let $F,G\in\oo[x]$ be monic separable polynomials, and let $F=F_1\cdots F_g$, $G=G_1\cdots G_{g'}$ be their factorization into a product of monic irreducible polynomials in $\oo[x]$. 
We say that $F$ and $G$ are \emph{Okutsu equivalent}, and we write $F\approx G$, if $g=g'$ and $F_s\approx G_s$ for all $1\le s\le g$, up to ordering.

An expression of the form, $F\approx P_1\cdots P_g$, with $P_1,\dots,P_g\in\oo[x]$ irreducible, is called an \emph{Okutsu factorization} of $F$. 
\end{definition}

Clearly, every $F\in\oo[x]$ admits a unique (up to $\approx$) Okutsu factorization. However, this concept is too weak for our purposes. For instance, if all factors of $F$ are Okutsu equivalent to $P$, then $F\approx P^g$ is an Okutsu factorization of $F$ which is unable to distinguish the true irreducible factors of $F$. 

\begin{definition}\label{OMfactorization}
Let $F\approx P_1\cdots P_g$ be an Okutsu factorization of a monic separable polynomial $F\in\oo[x]$. For each  $1\le s\le g$, let $F_s$ be the irreducible factor of $F$ which is Okutsu equivalent to $P_s$, and let $\t_s\in\ks$ be a root of $F_s$. 

We say that $F\approx P_1\cdots P_g$ is an \emph{OM factorization of }$F$ if 
\begin{equation}\label{distinguish}
 v(P_s(\t_s))>v(P_s(\t_t)),\quad \forall \, 1\le s\ne t\le g. 
\end{equation}
\end{definition}

\subsection{OM factorizations and OM representations}
This section is devoted to study basic properties of the OM factorizations and to find a characterization of the condition (\ref{distinguish}) in terms of OM representations of the factors of $F$, which facilitates the computation of these factorizations in practice.  

We denote by $\phi_i^\ty$, $\lambda_i^\ty$, $\psi_i^\ty$, $V_i^\ty$, etc. the data at the $i$-th level of a type $\ty$. Also, $\op{Rep}(\ty)$ denotes the set of representatives of the type $\ty$. 

\begin{lem-def}\label{equivTypes}
Let $\ty$, $\ty'$ be two strongly optimal types over $\oo$. The following conditions are equivalent:
\begin{enumerate}
\item[(a)] $\op{Rep}(\ty)=\op{Rep}(\ty')$.
\item[(b)] There exist representatives $\phi$, $\phi'$ of $\ty$, $\ty'$, respectively, such that $\phi\approx\phi'$.
\item[(c)] $\ord_\ty(F)=\ord_{\ty'}(F)$, for all polynomials $F\in\oo[x]$.
\end{enumerate}
When these conditions are satisfied, we say that the types $\ty$ and $\ty'$ are \emph{equivalent}.
\end{lem-def}

\begin{proof}
By Definition \ref{okequiv}, (a) and (b) are equivalent. Suppose that  $\ty$ and $\ty'$ admit a common representative $\phi$. By \cite[Thm. 3.9]{okutsu}, $[\phi_1^\ty,\dots,\phi_r^\ty]$ and $[\phi_1^{\ty'},\dots,\phi_{r'}^{\ty'}]$, are Okutsu frames of $\phi$; thus, $r=r'$ and 
the two types have the same Okutsu invariants and MacLane valuations $v_1,\dots,v_{r+1}$ \cite[Cor. 3.7]{okutsu}. Hence, the two types have the same Newton operators $N_{r+1}$, and (c) follows from Lemma \ref{previous},(2). Finally, since the representatives of $\ty$ are monic polynomials $\phi$ of degree $m_{r+1}$ such that $\ord_\ty(\phi)=1$, (c) trivially implies (a).
\end{proof}

If two strongly optimal types $\ty$, $\ty'$ of order $r$ are equivalent, then 
Lemmas \ref{previous}, \ref{representative} show that $\phi_i^\ty\approx\phi_i^{\ty'}$, for all $1\le i\le r$.  Since $\phi_i^\ty$ is a representative of $\op{Trunc}_{i-1}(\ty)$, the
 truncations of $\ty$ and $ \ty'$ of any order $0\le i\le r$ are equivalent too. 

By \cite[Thms. 3.5,3.9]{okutsu}, the mapping, $\ty\mapsto \op{Rep}(\ty)$, induces a 1-1  correspondence between equivalence classes of strongly optimal types and equi\-valence classes of monic irreducible separable polynomials in $\oo[x]$, under Okutsu equivalence. 

Let $F\in\oo[x]$ be a monic irreducible separable polynomial, and let $r$ be its Okutsu depth. We recall that an \emph{OM representation of $F$} is just an optimal type $\ty_F$ of order $r+1$, satisfying any of the following equivalent conditions:
\begin{itemize}
\item $\ty_F$ is $F$-complete; i.e. $\ord_{\ty_F}(F)=1$,
\item $\ty_F\mid F$ and $F\approx \phi_{r+1}^{\ty_F}$.
\end{itemize}
By Lemma-Definition \ref{equivTypes}, if $\ty_F$ and $\ty'_F$ are OM representations of $F$, the types $\op{Trunc}_r(\ty_F)$ and  $\op{Trunc}_r(\ty'_F)$ are equivalent.

\begin{definition}\label{icoin}
Let $F,G\in\oo[x]$ be monic irreducible separable polynomials of Okutsu depth $r_F$, $r_G$, and let $\ty_F$, $\ty_G$  be OM representations of $F$, $G$. Take $\phi_0^{\ty_F}=1=\phi_0^{\ty_G}$, by convention. The \emph{index of coincidence of $F$ and $G$} is the maximal index $0\le j\le \min\{r_F+1,r_G+1\}$, such that $\phi_j^{\ty_F}\approx \phi_j^{\ty_G}$. We denote this index by $i(F,G)$.
\end{definition}

The following properties of $i(F,G)$ are easy to check:
\begin{itemize}
\item $i(F,G)$ does not depend on the chosen OM representations $\ty_F$, $\ty_G$.
\item $i(F,G)$  depends only on the classes of $F$ and $G$ modulo $\approx$.
\item $F\approx G$ if and only if $i(F,G)=r_F+1=r_G+1$.
\end{itemize}
 
The next result is easily deduced from \cite[Prop. 3.5,(5)]{HN}.

\begin{proposition}\label{prop35}
Let $F,G\in\oo[x]$ be monic irreducible separable polynomials, and let $\t\in\ks$ be a root of $F$. Let $\ty$ be a type of order $i\ge 1$ over $\oo$, such that $\ty\mid F$ and $\op{Trunc}_{i-1}(\ty)\mid G$. Let $\lambda(G)$ be the slope of (the one-sided polygon) $N_i(G)$.Then, 
$$
v(G(\t))/\deg G\ge (V_i+\min\{|\lambda_i|,|\lambda(G)|\})/(m_ie_0\cdots e_{i-1}),
$$
and equality holds if and only if $\ty\nmid G$. 
\end{proposition}

\begin{lemma}\label{divideG}
Let $F,G\in\oo[x]$ be monic irreducible separable polynomials, and let $\t\in\ks$ be a root of $F$. Let $\ty$ be a strongly optimal type of order $i$ over $\oo$, such that $\ty\mid F$. Then, the following conditions are equivalent
\begin{enumerate}
\item[(a)] $\ty\mid G$.
\item[(b)] $i(F,G) > i$.
\item[(c)] $v(G(\t))/\deg G > V_{i+1}/(m_{i+1}e_0\cdots e_i)=v(\phi_i(\t))/m_i$.
\end{enumerate}
\end{lemma}

\begin{proof}
By \cite{HN,algorithm}, the type $\ty$ may be extended to an OM representation $\ty_F$ of $F$. If $\ty\mid G$, it may be extended to an OM representation $\ty_G$ of $G$ too; thus, $\phi_{i+1}^{\ty_F}\approx\phi_{i+1}^{\ty_G}$, because they are both a representative of $\ty$. Thus, (a) implies (b).
Conversely, let $\ty_G$ be an arbitrary OM representation of $G$, and suppose $\phi_{i+1}^{\ty_F}\approx\phi_{i+1}^{\ty_G}$. This implies that $\phi_{i+1}^{\ty_G}$ is a representative of $\ty$; thus, the types $\ty$ and $\op{Trunc}_i(\ty_G)$ are equivalent. By the last item of Definition \ref{optimal}, $0<\ord_{\ty_G}(G)\le \ord_{\op{Trunc}_i(\ty_G)}(G)=\ord_\ty(G)$. Therefore, (a) and (b) are equivalent.

Let us now show that (a) and (c) are equivalent.
If $\psi_0^\ty\nmid \overline{G}$, then $v(G(\t))=0$ and $\ty\nmid G$; thus (a) and (c) are both false in this case.
Suppose $\psi_0^\ty\mid \overline{G}$, and let $1\le j\le i+1$ be maximal such that $\op{Trunc}_{j-1}(\ty)\mid G$. 
Let $\la_{i+1}$ be the slope of $N_{i+1}(F)$. The Newton polygon $N_j^-(G)$ with respect to $\ty$ has a positive length by Lemma \ref{previous}; let $\lambda(G)\in\Q_{<0}$ be its slope. By Proposition \ref{prop35},
$$
v(G(\t))/\deg G\ge(V_j+\min\{|\lambda_j|,|\lambda(G)|\})/(m_je_0\cdots e_{j-1}),
$$   
and equality holds if $j\le i$, because $\op{Trunc}_{j}(\ty)\nmid G$. If $\ty\mid G$, then $j=i+1$,  and 
$v(G(\t))/\deg G>V_{i+1}/(m_{i+1}e_0\cdots e_i)$. If $\ty\nmid G$, then $j\le i$, and 
$$
\dfrac{v(G(\t))}{\deg G}\le\dfrac{V_j+|\lambda_j|}{m_je_0\cdots e_{j-1}}=\dfrac{v(\phi_j(\t))}{m_j}\le \dfrac{v(\phi_i(\t))}{m_i}=\dfrac{V_{i+1}}{m_{i+1}e_0\cdots e_i},
$$   
by Lemma \ref{previous} and the properties (\ref{frame}) of the Okutsu polynomials.
\end{proof}

\begin{lemma}\label{tPF}
Let $P,Q\in\oo[x]$ be monic irreducible separable polynomials such that $P\approx Q$. Let $\ty=(\psi_0;(\phi_1,\lambda_1,\psi_1);\cdots (\phi_r,\lambda_r,\psi_r))$ be a strongly optimal type admitting $P$ as a representative. 
Then, there exist unique data $(\lambda_Q,\psi_Q)$ (or $(-\infty,\hbox{---})$, if $P=Q$), such that $\ty_Q:=(\psi_0;(\phi_1,\lambda_1,\psi_1);\cdots (\phi_r,\lambda_r,\psi_r);(P,\lambda_Q,\psi_Q))$, is an OM representation of $Q$.
\end{lemma}

\begin{proof}
Since $Q$ is also a representative of $\ty$, we have $\ord_\ty(Q)=1$, and the Newton polygon $N_{r+1}^-(Q)$ with respect to $\ty$ and $P$ has length one by Lemma \ref{previous}. Let $\lambda_Q\in\Z\cup\{-\infty\}$ be the slope of this polygon. 
If $\lambda_Q\ne-\infty$ (i.e. $P\ne Q$), the residual polynomial $R_{\lambda_Q,r+1}(Q)$ has degree one; let $\psi_Q$ be the monic polynomial obtained by dividing this polynomial by its leading coefficient. By construction, $\ty_Q\mid Q$. By the last item of Definition \ref{optimal}, $\ord_{\ty_Q}(Q)\le \ord_\ty(Q)=1$; thus, $\ord_{\ty_Q}(Q)=1$, so that $\ty_Q$ is an OM representation of $Q$. Also, once we choose $P$ as a representative of $\ty$, the condition $\ty_Q\mid Q$ uniquely determines these data $(\lambda_Q,\psi_Q)$.   
\end{proof}

The computation of an Okutsu factorization $F\approx P_1\cdots P_g$ of a monic separable polynomial $F$ is equivalent to the computation of a family $\ty_{F_1},\dots,\ty_{F_g}$ of OM representations of the irreducible factors of $F$. In fact, from the Okutsu factors $P_1,\dots,P_g$ and strongly optimal types $\ty_1,\dots,\ty_g$ such that each $\ty_s$ admits $P_s$ as a representative, we may construct the OM representations of $F_1,\dots,F_g$, as shown in Lemma \ref{tPF}. Conversely, from the family  $\ty_{F_1},\dots,\ty_{F_g}$ we may take $P_s:=\phi_{r_s+1}^{\ty_{F_s}}\approx F_s$, as Okutsu factors, where $r_s$ is the Okutsu depth of $F_s$.

We now describe the property of being an OM factorization in terms of the family  $\ty_{F_1},\dots,\ty_{F_g}$ of OM representations.

\begin{proposition}\label{OMOM}
Let $F\in\oo[x]$ be a monic separable polynomial and $F_1,\dots,F_g\in\oo[x]$ its monic irreducible factors, with Okutsu depth $r_1,\dots,r_g$, respectively.
Let $\ty_{F_1},\dots,\ty_{F_g}$ be a family of OM representations of the factors, and let $P_s:=\phi_{r_s+1}^{\ty_{F_s}}$. Let $I$ be the set of ordered pairs $(s,t)$ of indices such that  $i(F_s,F_t)=r_s+1$, and for each $(s,t)\in I$, let $\lambda_{s,t}$ be the slope of $N_{r_s+1,\ty_{F_s}}(F_t)$.
Then, the Okutsu factorization $F\approx P_1\cdots P_g$ is an OM factorization of $F$ if and only if 
\begin{equation}\label{distinguishOM}
|\lambda_{s,s}|>|\lambda_{s,t}|, \ \forall\,(s,t)\in I,\ s\ne t.
\end{equation}
\end{proposition}

\begin{proof}
Denote $\ty_s:=\op{Trunc}_{r_s}(\ty_{F_s})$, and choose a root $\t_s\in\ks$ of $F_s$,  for each $1\le s\le g$. Let $(s,t)$ be an ordered pair of indices, $1\le s,t\le g$. Suppose $i(F_s,F_t)=r_s+1$. Then, Lemma \ref{divideG} shows that $\ty_s\mid F_t$, and 
$$
v(P_s(\t_t))=\left(V_{r_s+1}^{\ty_s}+|\lambda_{s,t}|\right)/e(F_s),
$$ 
by Lemma \ref{previous}. Suppose now $i:=i(F_s,F_t)\le r_s$. Since $i(P_s,F_t)=i(F_s,F_t)=i$, Lemma \ref{divideG} shows that $\op{Trunc}_i(\ty_{F_t})\nmid P_s$. By Proposition \ref{prop35},
\begin{align*}
v(P_s(\t_t))&\ = \dfrac{m^{\ty_s}_{r_s+1}}{m_i}\,\dfrac{V_i+\min\{|\lambda_i^{\ty_s}|,|\lambda_i^{\ty_t}|\}}{e_0\cdots e_{i-1}}\le\dfrac{m^{\ty_s}_{r_s+1}}{m_i}\,\dfrac{V_i+|\lambda_i^{\ty_s}|}{e_0\cdots e_{i-1}}\\&\ =\dfrac{m^{\ty_s}_{r_s+1}}{m^{\ty_s}_{i+1}}\,\dfrac{V^{\ty_s}_{i+1}}{e^{\ty_s}_0\cdots e^{\ty_s}_{i}}\le \dfrac{V_{r_s+1}^{\ty_s}}{e(F_s)},
\end{align*}
the last inequality by the explicit formulas of $V_j$ in section \ref{secOkutsu}. Hence, the condition (\ref{distinguish}) is equivalent to (\ref{distinguishOM}).
\end{proof}

\begin{definition}\label{faithful}
Let $F\in\oo[x]$ be a monic separable polynomial and $F_1,\dots,F_g\in\oo[x]$ the monic irreducible factors of $F$. We say that a family $\ty_{F_1},\dots,\ty_{F_g}$ of OM representations of the factors \emph{faithfully represents $F$} if any of the two following equivalent conditions is satisfied:
\begin{enumerate}
\item[(a)] $\ty_{F_s}\nmid F_{t}, \quad \forall\,1\le s\ne t\le g$.
\item[(b)] $\ord_{\ty_{F_s}}(F)=1, \quad \forall\,1\le s\le g$.
\end{enumerate}
\end{definition}
By construction, $\ord_{\ty_{F_s}}(F_s)=1$; hence, the conditions (a) and (b) are equivalent because $\ord_{\ty_{F_s}}(F)=\sum_{1\le t\le g}\ord_{\ty_{F_s}}(F_t)$. 

\begin{corollary}\label{faithful2}
With the notation in Proposition \ref{OMOM}, if $F\approx P_1\cdots P_g$ is an OM factorization, then the family  $\ty_{F_1},\dots,\ty_{F_g}$ faithfully represents $F$. 
\end{corollary}

\begin{proof}
If $\ty_{F_s}\mid F_t$, then $F_t$ is a polynomial of type $\ty_{F_s}$ (Lemma \ref{oftype}) and this implies $\lambda_{s,t}=\lambda_{s,s}$ (Definition \ref{type}).
\end{proof}

Finally, we show that any family of OM representations that faithfully represents a polynomial $F$, leads immediately to an OM factorization of $F$.

\begin{lemma}\label{construction}
Let $F\in\oo[x]$ be a monic separable polynomial and  $\ty_{F_1},\dots,\ty_{F_g}$ a family of OM representations of the irreducible factors of $F$, that faithfully represents $F$. Then, if we take arbitrary representatives $Q_1,\dots,Q_g$ of these types, we get an OM factorization, $F\approx Q_1\cdots Q_g$, of $F$. 
\end{lemma}

\begin{proof}
We keep the notation from Proposition \ref{OMOM}. Consider an index $1\le s \le g$. All data $e_j,f_j,h_j,V_j$ we are going to use correspond to the type $\ty_{F_s}$.
Since $\ord_{\ty_{F_s}}(F_s)=1$, the Newton polygon $N_{r_s+2,\ty_{F_s}}^-(F_s)$ has length one and slope $-h_s\in\Z_{<0}\cup\{-\infty\}$. By \cite[Thm. 3.1]{HN},
$$
v(Q_s(\t_s))=(V_{r_s+2}+h_s)/e(F_s)=(V_{r_s+1}+|\lambda_{s,s}|+h_s)/e(F_s),
$$ 
the last equality by the recurrence $V_{r_s+2}=e_{r_s+1}f_{r_s+1}(e_{r_s+1}V_{r_s+1}+h_{r_s+1})$, in section \ref{secOkutsu}, having in mind that $e_{r_s+1}=f_{r_s+1}=1$ and $h_{r_s+1}=|\lambda_{s,s}|$.

For all $t\ne s$, we have $\ty_{F_s}\nmid F_t$. If $\ty_s\mid F_t$, then Proposition \ref{prop35} shows that
$$
v(Q_s(\t_t))=(V_{r_s+1}+\min\{|\lambda_{s,s}|,|\lambda_{s,t}|\})/e(F_s)<v(Q_s(\t_s)).
$$   
If $\ty_s\nmid F_t$, then $i:=i(F_s,F_t)=i(Q_s,F_t)\le r_s$, and $\op{Trunc}_{i}(\ty_{F_t})\nmid Q_s$, by Lemma \ref{divideG}. Thus, $v(Q_s(\t_t))\le V_{r_s+1}/e(F_s)<v(Q_s(\t_s))$, as in the proof of Proposition \ref{OMOM}. 
\end{proof}

Let us see an example. Take $a,b\in\oo$ such that $v(ab)=0$ and consider
$$
F_1=x+\pi+\pi^2+\pi^4a, \quad F_2=x+\pi+\pi^3+\pi^4b, \quad F=F_1F_2.
$$
The Okutsu factorizations, $F\approx x^2\approx x(x+\pi)$, are not OM factorizations of $F$, because they lead both to
$\ty_{F_1}=(y;(x,-1,y+1))\mid F_2$. 

The Okutsu factorization $F\approx (x+\pi)^2$ leads to a family of OM representations that faithfully represents $F$, because these Okutsu factors are sufficiently close to the true factors to distinguish them:
$$
\ty_{F_1}=(y;(x+\pi,-2,y+1))\nmid F_2,\quad
\ty_{F_2}=(y;(x+\pi,-3,y+1))\nmid F_1.$$  
Let us choose as representatives of the above types $\ty_{F_1}$, $\ty_{F_2}$, the polynomials $Q_1=x+\pi+\pi^2$, $Q_2=x+\pi+\pi^3$. By Lemma \ref{construction},
$F\approx Q_1Q_2$ is an OM factorization. The new OM representations of $F_1$, $F_2$ determined by $Q_1$, $Q_2$ are:
$$
\ty_{F_1}=\left(y;(x+\pi+\pi^2,-4,y+\overline{a})\right),\quad \ty_{F_2}=\left(y;(x+\pi+\pi^3,-4,y+\overline{b})\right).
$$  

The Montes algorithm computes a family $\ty_{F_1},\dots,\ty_{F_g}$ of OM representations faithfully representing $F$, and derives from it an OM factorization $F\approx P_1\cdots P_g$, as indicated in Lemma \ref{construction} (cf. section \ref{secAlgo}). This is the starting point for the fast computation of an approximate factorization of $F$ with a prescribed precision, by means of the single-factor algorithm \cite{GNP}.    

\subsection{Polynomials having the same OM factorizations} 
The aim of this section is to prove Theorem \ref{precision}, where we find the least precision $\nu$ such that two polynomials congruent modulo $\m^\nu$ have the same OM factorizations. To this end, we need a result similar in spirit to Lemma \ref{technical}.

\begin{lemma}\label{technical2}
Let $\ty$ be a strongly optimal type of order $i-1\ge 0$, and $\phi\in \oo[x]$ a representative of $\ty$. 
Let $F\in\oo[x]$ be a monic polynomial such that $\ell(N_i^-(F))>1$. 
Let $\sm$ be the first side (from left to right) of $N_i^-(F)$ and let $\lm$ be its slope. Let $u>u'$ be the ordinates of the end points of $\sm$. If $\ell(\sm)=1$, let $\sn$ be the second side of $N_i^-(F)$ and let $\ln$ be its slope. Then,
$$
\delta(F)\ge\begin{cases}u,&\mbox{ if }\ell(\sm)>1,\\
u'+|\ln|,&\mbox{ if }\ell(\sm)=1.
\end{cases}
$$
\end{lemma}

\begin{proof}
Let $F=F_1\cdots F_g$ be the factorization of $F$ into a product of monic irreducible polynomials in $\oo[x]$. For all $1\le j\le i$, $1\le s\le g$, denote:\medskip

$\ty_{j-1}:=\op{Trunc}_{j-1}(\ty)$,

$\ell_{j,s}:=\ell(N_j^-(F_s))=\ord_{\ty_{j-1}}(F_s)$, the abscissa of the right end point of $N_j^-(F_s)$,

$u_{j,s}:=$ the ordinate of the left end point of $N_j^-(F_s)$,\medskip

\noindent By \cite[Lem. 2.17]{HN}, the right end point of $N_j^-(F_s)$ is $(\ell_{j,s},v_j(F_s))$. If $\ty_{j-1}\mid F_s$, then Lemma \ref{oftype} shows that $\deg F_{s}=m_j\ell_{j,s}$. In particular, $v_j(F_s)=\ell_{j,s}V_j$ and $u_{j,s}=\ell_{j,s}(V_j+|\lambda_{j,s}|)$, where $\lambda_{j,s}$ is the slope of $N_j^-(F_s)$. If $\ty_{j-1}\nmid F_s$, then $\ell_{j,s}=0$ and $u_{j,s}=v_j(F_s)$. 

By the theorem of the product (\ref{product}), $u=u_{i,1}+\cdots+u_{i,g}$, and there exists an irreducible factor $F_{s_0}$ such that $N_i^-(F_{s_0})$ is one-sided of slope $\lm$. Since $F_{s_0}$ is a polynomial of type $\ty$, Lemma \ref{oftype} shows that $\deg F_{s_0}=m_j\ell_{j,s_0}$, for all $1\le j\le i$. \medskip

\noindent{\bf Claim. }For all $s\ne s_0$, we have $v(\res(F_s,F_{s_0}))\ge u_{i,s}$.\medskip
 
In fact, suppose first that $\ty\nmid F_s$. Let $0\le j< i$ be the first level such that $\ty_j\nmid F_s$. For all $j<k\le i$, the Newton polygon
$N_k^-(F_s)$ is the single point $(0,v_k(F_s))$. By the definition of the MacLane valuations,
$u_{i,s}=v_i(F_s)=e_{i-1}\cdots e_{j+1}v_{j+1}(F_s)$. If $j=0$, then $v_1(F_s)=0$ and we deduce that $u_{i,s}=0$. If $0<j<i$, then $\ty_{j-1}\mid F_s$, and $v_{j+1}(F_s)=e_j(v_j(F_s)+\ell_{j,s}\min\{|\lambda_{j,s}|,|\lambda_{j}|\})$, by the definition of $v_{j+1}$. Hence,
$$
u_{i,s}=e_{i-1}\cdots e_j\ell_{j,s}\left(V_j+\min\{|\lambda_{j,s}|,|\lambda_{j}|\}\right)
\le  e_{i-1}\cdots e_j\ell_{j,s}\left(V_j+|\lambda_{j,s}|\right).
$$
On the other hand, Lemma \ref{resultant} applied to the type $\ty_{j-1}$ shows that
\begin{align*}
 v(\res(F_s,F_{s_0}))&\;\ge f_0\cdots f_{j-1}\ell_{j,s}\ell_{j,s_0}\left(V_j+\min\{|\lambda_{j,s}|,|\lm|\}\right)\\&\;= m_j\ell_{j,s}\ell_{j,s_0}\dfrac{V_j+|\lambda_{j,s}|}{e_0\cdots e_{j-1}}=\deg(F_{s_0})\ell_{j,s}\dfrac{V_j+|\lambda_{j,s}|}{e_0\cdots e_{j-1}}\\
&\;\ge m_i\ell_{j,s}\dfrac{V_j+|\lambda_{j,s}|}{e_0\cdots e_{j-1}}\ge  e_{i-1}\cdots e_j\ell_{j,s}\left(V_j+|\lambda_{j,s}|\right)\ge u_{i,s}.
\end{align*}

If  $\ty\mid F_s$, we have directly 
$u_{i,s}=\ell_{i,s}(V_i+|\lambda_{i,s}|)\le v(\res(F_s,F_{s_0}))$, 
by Lemma \ref{resultant} applied to the type $\ty$.
This ends the proof of the Claim.\medskip

From now on, we denote $\rho_{s,t}:=v(\res(F_s,F_{t}))$. We are ready to deduce the lemma from the Claim and the equality:
$$
\delta(F)=\sum\nolimits_{1\le s\le g}\delta(F_s)+\sum\nolimits_{1\le s,t\le g}\rho_{s,t}.
$$
Suppose first that there is at least one $F_{s_1}\ne F_{s_0}$, such that $\ty\mid F_{s_1}$ and $\la_{i,s_1}=\lm$. In this case, the Claim shows by symmetry that 
$\rho_{s_1,s_0}\ge u_{i,s_0}$; hence,
$$
\delta(F)\ge 2\rho_{s_1,s_0}+\sum\nolimits_{s\ne s_0,s_1}\rho_{s,s_0}\ge \sum\nolimits_{1\le s\le g}u_{i,s}=u.
$$

Suppose now that for all $F_s\ne F_{s_0}$, such that $\ty\mid F_s$, we have $\lambda_{i,s}\ne\lm$. In this case, $\ell_{i,s_0}=\ell(\sm)$ and $u=u'+\ell_{i,s_0}|\lm|$. If $\ell_{i,s_0}>1$, we have $\deg F_{s_0}=m_i\ell_{i,s_0}\ge 2m_i$, so that the Okutsu depth of $F_{s_0}$ is greater than or equal to $i$. Lemma \ref{delta0props} shows that $2\delta(F_{s_0})/\deg F_{s_0}\ge u_{i,s_0}/(e_0\cdots e_{i-1})$, and we deduce that $\delta(F_{s_0})\ge m_iu_{i,s_0}/(e_0\cdots e_{i-1})\ge u_{i,s_0}$. Hence,
$$
\delta(F)\ge \delta(F_{s_0})+\sum\nolimits_{s\ne s_0}\rho_{s,s_0}\ge \sum\nolimits_{1\le s\le g}u_{i,s}=u.
$$ 
Finally, suppose that $\ell_{i,s_0}=\ell(\sm)=1$. In this case,  $\ord_\ty(F_{s_0})=\ell_{i,s_0}=1$, $v_i(F_{s_0})=\ell_{i,s_0}V_i=V_i$, and $u_{i,s_0}=V_i+|\lm|$. Since $\ell(N_i^-(F))>1$, this polygon has at least a second side $\sn$ of slope $\ln$. Let $I$ be the set of all indices $1\le t\le g$ such that $N_i^-(F_t)$ has slope $\ln$. By the Claim, for all $t\in I$, we have
$$
2\rho_{t,s_0}\ge 2u_{i,t}=\ell_{i,t}(V_i+|\ln|)+u_{i,t}
\ge v_i(F_{s_0})+|\ln|+u_{i,t},
$$
so that
\begin{align*}
\delta(F)\ge &\;2\sum\nolimits_{t\in I}\rho_{t,s_0}+\sum\nolimits_{s\not\in I\cup\{s_0\}}\rho_{s,s_0} \ge  
v_i(F_{s_0})+|\ln|+\sum\nolimits_{s\ne s_0}u_{i,s}\\=&\;|\ln|+\left(\sum\nolimits_su_{i,s}\right)-|\lm|=|\ln|+u'.
\end{align*}
\end{proof}

\begin{figure}\caption{Newton polygon $N_i^-(F)$ in the context of Lemma \ref{technical2}.}\label{figPrecision}
\setlength{\unitlength}{5mm}
\begin{center}
\begin{picture}(20,10)
\put(.85,9.15){$\bullet$}\put(1.85,5.15){$\bullet$}
\put(3.85,3.15){$\bullet$}\put(6.85,1.65){$\bullet$}
\put(0,.3){\line(1,0){9}}\put(.9,6.3){\line(1,0){.2}}
\put(1,9.3){\line(1,-4){1}}\put(1,9.32){\line(1,-4){1}}
\put(1,6.3){\line(1,-1){3}}\put(2,5.32){\line(1,-1){2}}
\put(4,3.3){\line(2,-1){3}}\put(4,3.32){\line(2,-1){3}}
\put(1,-.7){\line(0,1){10.5}}
\put(1.6,7.8){\begin{footnotesize}$\sm$\end{footnotesize}}
\put(3.2,4.3){\begin{footnotesize}$\sn$\end{footnotesize}}
\put(.3,9.3){\begin{footnotesize}$u$\end{footnotesize}}
\put(.2,5.2){\begin{footnotesize}$u'$\end{footnotesize}}
\put(-2.2,6.2){\begin{footnotesize}$u'+|\ln|$\end{footnotesize}}
\put(.5,-.4){\begin{footnotesize}$0$\end{footnotesize}}
\put(1.9,-.4){\begin{footnotesize}$1$\end{footnotesize}}
\put(6.9,-.4){\begin{footnotesize}$\ell$\end{footnotesize}}
\put(-.8,1.65){\begin{footnotesize}$v_i(F)$\end{footnotesize}}
\put(-1.5,7.3){\begin{footnotesize}$\delta(F)\!+\!1$\end{footnotesize}}
\multiput(7,.2)(0,.25){7}{\vrule height2pt}
\multiput(2,.2)(0,.25){21}{\vrule height2pt}
\multiput(.85,1.8)(.25,0){25}{\hbox to 2pt{\hrulefill}}
\multiput(.85,7.4)(.25,0){31}{\hbox to 2pt{\hrulefill}}
\multiput(.85,5.4)(.25,0){5}{\hbox to 2pt{\hrulefill}}

\put(12.85,5.15){$\bullet$}
\put(14.85,3.15){$\bullet$}\put(18.85,1.8){$\bullet$}
\put(12,.3){\line(1,0){9}}
\put(13,5.3){\line(1,-1){2}}\put(13,5.32){\line(1,-1){2}}
\put(15,3.3){\line(3,-1){4}}\put(15,3.32){\line(3,-1){4}}
\put(13,-.7){\line(0,1){10.5}}
\put(14,4.5){\begin{footnotesize}$\sm$\end{footnotesize}}
\put(12.3,5.3){\begin{footnotesize}$u$\end{footnotesize}}
\put(12.5,-.4){\begin{footnotesize}$0$\end{footnotesize}}
\put(18.9,-.4){\begin{footnotesize}$\ell$\end{footnotesize}}
\put(11.2,1.8){\begin{footnotesize}$v_i(F)$\end{footnotesize}}
\put(10.5,6.3){\begin{footnotesize}$\delta(F)\!+\!1$\end{footnotesize}}
\multiput(19,.2)(0,.25){7}{\vrule height2pt}
\multiput(12.85,2)(.25,0){25}{\hbox to 2pt{\hrulefill}}
\multiput(12.85,6.4)(.25,0){31}{\hbox to 2pt{\hrulefill}}
\end{picture}
\end{center}
\end{figure}
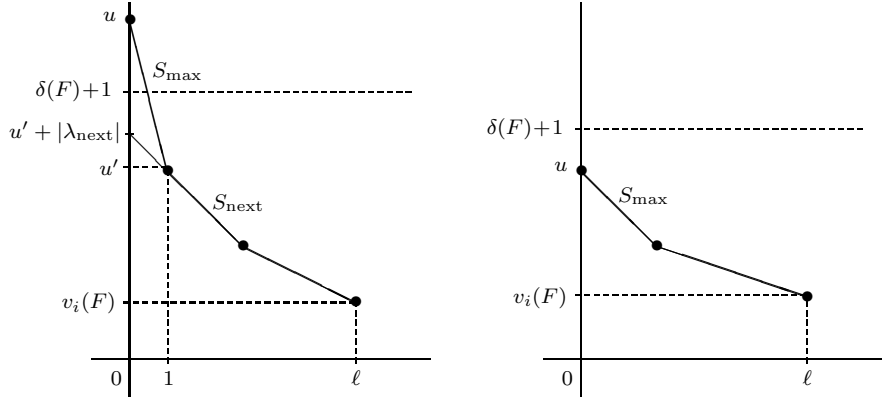

\begin{remark}\label{infinity}\rm
In Lemma \ref{technical2}, if $\phi$ divides $F$, then we understand that $\sm$ is a side of slope $\lm=-\infty$, and $u=\infty$ \cite[Sec. 1.1]{HN}. The statement  of the lemma and all arguments in the proof remain vaid in this case.
\end{remark}

It is easy to construct examples showing that the inequalities of Lemma \ref{technical2} are sharp. For instance, $F(x)=x^2+\pi^\nu$ has $u=\delta=\nu$ (if $v(2)=0$); while $F(x)=(x+\pi^\nu)(x+\pi)$ has $u'=|\ln|=1$ and $\delta=2$, if $\nu>1$.  
 
\begin{theorem}\label{precision}
Let $F,G\in\oo[x]$ be monic separable polynomials such that $F\equiv G\md{\m^{\delta(F)+1}}$. Then, $F\approx G$ and any OM factorization $F\approx P_1\cdots P_g$ of $F$ is also an OM factorization $G\approx P_1\cdots P_g$ of $G$.
\end{theorem}

\begin{proof}
Let $F_1,\dots,F_g$ be the monic irreducible factors of $F$, ordered so that $F_s\approx P_s$, for all $1\le s\le g$. Our aim is to attach to every $P_s$ an irreducible factor $G_s$ of $G$, such that $G_s\approx P_s$ and either (\ref{distinguish}) or (\ref{distinguishOM}) are satisfied for the pair $P_s,G$. 

Let us fix an index $1\le s\le g$. Let $r$ be the Okutsu depth of $P_s$ and let
$$
\ty_{F_s}=
\left(\psi_{0};(\phi_{1},\lambda_{1},\psi_{1});\cdots;(\phi_{r},\lambda_{r},\psi_{r});(P_s,\lambda_{F_s},\psi_{F_s})\right),
$$
be the OM representation of $F_s$ determined by $P_s$, satisfying $\ty_{F_s}\nmid F_{t}$ for all $t\ne s$. We admit exact OM representations in which $\la_{F_s}=-\infty$ and $\psi_{F_s}$ is not defined.

Consider the strongly optimal type $\ty:=\op{Trunc}_r(\ty_{F_s})$. Since $F_s\approx P_s$, the polynomial $F_s$ is a representative of $\ty$ too; thus, $\ord_{\ty}(F_s)=1$. 
The proof of the theorem requires different arguments according to $\ord_\ty(F)=1$ or $\ord_\ty(F)>1$.\medskip

\noindent{\bf Case $\operatorname{\bf ord}_{\mathbf{\ty}}\mathbf{(F)=1}$.} Since $1=\ord_\ty(F)=\sum_{1\le t\le g}\ord_\ty(F_t)$, we have $\ord_\ty(F_{t})=0$, for all $t\ne s$. By Definition \ref{okequiv}, $F_{t}\not\approx F_s\approx P_s$, for all $t\ne s$.  

For monic polynomials $P,Q\in\oo[x]$ of Okutsu depth zero we have $P\approx Q$ if and only if $\overline{P}=\overline{Q}$. Thus, if $r=0$, then $\overline{P_s}=\overline{F_s}=\psi_0$ is coprime to $\overline{F_{t}}$, for all $t\ne s$. By hypothesis, $\overline{G}=\overline{F}=\overline{F_1}\cdots\overline{F_g}$; thus, by Hensel's lemma, $G$ has a unique irreducible factor (say) $G_s$, such that $\overline{G_s}=\psi_0$ is coprime to $\overline{G}/\overline{G_s}$. Hence, $P_s\approx G_s$, $v(P_s(\t_s))>0$ and $v(P_s(\t))=0$, for any choice of roots $\t_s,\t\in\ks$ of $G_s$ and $G/G_s$, respectively. Thus, (\ref{distinguish}) is satisfied for the pair $P_s,G$.  

If $r>0$, we may consider $\ty_{r-1}:=\op{Trunc}_{r-1}(\ty)$. By the last item of Definition \ref{optimal}, $\ord_{\ty_{r-1}}(F)\ge \ord_{\ty_{r-1}}(F_s)\ge e_rf_r\ord_\ty(F_s)>1$. Since $\ty\mid F_s$, the polygon $N_r^-(F_s)$ is one-sided of slope $\lambda_r$ and it has length $\ord_{\ty_{r-1}}(F_s)>1$, by Lemma \ref{previous}. By (\ref{product}), $N_r^-(F)$ has a side $S$ of slope $\lambda_r$ and length $\ell(S)>1$, where $\ell(S)$ is the length of the projection of $S$ to the horizontal axis. 

We now apply Lemma \ref{technical2} to the pair $\ty_{r-1}$, $F$. If $\ell(\sm)=1$, then $S\ne \sm$, because $\ell(S)>1$. In any case,  Lemma \ref{technical2} shows that $\delta(F)+1$ is greater than the ordinate of the point of the vertical axis lying on the line determined by $S$. By Lemma \ref{belowline}, the Newton polygon $N_r^-(G)$ has a side of slope $\lambda_r$ and $R_r(G)=R_r(F)$; thus,  $\ord_\ty(G):=\ord_{\psi_r}R_r(G)=\ord_{\psi_r}R_r(F)=:\ord_\ty(F)=1$. Hence, there is a unique irreducible factor (say) $G_s$ of $G$, such that  $\ord_\ty(G_s)=1$, and $\ord_\ty(G_0)=0$, for any other irreducible factor $G_0$ of $G$. By Lemma \ref{oftype}, $\deg G_s=m_{r_s+1}\ord_{\ty}(G_s)=m_{r_s+1}$; thus, $G_s$ is a representative of $\ty$, and $G_s\approx P_s$. Finally, the set $I$ in Proposition \ref{OMOM} contains only the pair $(s,s)$, so that (\ref{distinguishOM}) is trivially satisfied. \medskip

\noindent{\bf Case $\operatorname{\bf ord}_{\mathbf{\ty}}\mathbf{(F)>1}$. }Since $F_s\approx P_s$ is a representative of $\ty$, we have $\ord_\ty(F_s)=1$, so that $N_{r+1}(F_s)$ has length one and slope $\lambda_{s,s}$, in the notation from Proposition \ref{OMOM}.
Since $F\approx P_1\cdots P_g$ is an OM factorization, (\ref{distinguishOM}) holds; this implies that  $N_{r+1}^-(F)$ indeed has a side $\sm$ of slope $\lm=\lambda_{s,s}$ and end points $(0,u)$ and $(1,u')$, by the theorem of the product (\ref{product}). 

We now apply Lemma \ref{technical2} to the pair $\ty$, $F$. Arguing as before, $N_{r+1}^-(G)$ coincides with $N_{r+1}^-(F)$, except for, eventually, the ordinate $u$ of the point of abscissa zero (see Figure \ref{figPrecision}). Thus, $N_{r+1}^-(G)$ has also a first side $\sm(G)$ of length one and slope $\lambda_{s,s}(G)$, with  $|\lambda_{s,s}(G)|>|\lambda_{s,t}|$, for all $t$ such that $\ty\mid F_t$. The equality of the Newton polygons (up to the first side) and the theorem of the product, show that all irreducible factors $G_0\ne G_s$ of $G$, which are divisible by $\ty$, have $N_{r+1}(G_0)$ one-sided of slope $\lambda_{s,t}$ for some $t\ne s$. Hence, (\ref{distinguishOM}) is satisfied for $P_s,G$ as well.       
\end{proof}

\section{The factorization algorithm of Ore, MacLane and Montes}\label{secAlgo}
Let us go back to the global setting of the Introduction. Let $A$ be a Dedekind domain whose field of fractions $K$ is a global field. Let $L/K$ be a finite separable extension and $B$ the integral closure of $A$ in $L$. Let $\t\in L$ be a primitive element of $L/K$, with minimal polynomial $F(x)\in A[x]$.

Let $\p$ be a non-zero prime ideal of $A$, $v:=v_\p$ the canonical $\p$-adic valuation, $K_\p$ the completion of $K$ at $\p$, and $\oo_\p$ the valuation ring of $K_\p$. We denote by $\F=A/\p$ the residue field of $\p$. We fix a local generator $\pi$ of $\p$; that is, an element $\pi\in A$, whose image in the local ring $A_\p$ generates the maximal ideal. If $A$ is a principal domain, we assume moreover that $\p=\pi A$. 

The Montes algorithm was developed by J. Montes in his PhD thesis, inspired by the ideas of \O. Ore and S. MacLane. It is fully described in \cite{algorithm}, in terms of the theoretical background developed in \cite{HN}. For a short review the reader may check the survey \cite{Nsurvey}.

The algorithm is based on four routines: {\tt Factorization}, {\tt Newton}, {\tt ResidualPo\-lynomial} and {\tt Representative}. Let us briefly review them.\bigskip

\noindent{\tt Routine Factorization($\mathcal{F}$,\,$\varphi$)}\vskip 1mm

\noindent INPUT:

\noindent $-$ A finite field $\mathcal{F}$.

\noindent $-$ A monic polynomial $\varphi(y)\in \mathcal{F}[y]$.\medskip

\noindent OUTPUT:

\noindent $-$ The factorization of $\varphi(y)$ into a product of irreducible polynomials of $\mathcal{F}[y]$.\bigskip

\noindent{\tt Routine Newton($\ty$,\,$\omega$,\,$g$)}\vskip 1mm

\noindent INPUT:

\noindent $-$ A type $\ty$ over $A$, of order $i-1\ge 0$, and a representative $\phi\in A[x]$ of $\ty$.

\noindent $-$ A non-negative integer $\omega$.

\noindent $-$ A non-zero polynomial $g(x)\in K[x]$.\medskip

\noindent Compute the first $\omega+1$ coefficients $a_0(x),\dots,a_\omega(x)$ of the canonical $\phi$-expansion of $g(x)$ and the Newton polygon $N$ of the set of points $(s,v_i(a_s\phi^s))$, for $0\le s\le\omega$.\medskip

\noindent OUTPUT:

\noindent $-$ $N_i(g):=N$ is the $i$-th order Newton polygon of $g$ with respect to the pair $(\ty,\phi)$.

\begin{definition}\label{sla}
Let $\lambda\in\Q_{<0}$ and $N$ a Newton polygon. We define the \emph{$\lambda$-compo\-nent} of $N$ to be $S_{\lambda}(N):=\{(x,y)\in N\mid y-\lambda x\mbox{ is minimal}\}$. 
If $N$ has a side $S$ of slope $\lambda$, then $S_{\lambda}(N)=S$; otherwise, $S_\lambda(N)$ is a vertex of $N$ (see Figure \ref{figComponent}).  
\end{definition}

\begin{figure}\caption{$\lambda$-component of a polygon. $L_\lambda$ is the line of slope $\lambda$ having first contact with the polygon from below.}\label{figComponent}
\begin{center}
\setlength{\unitlength}{5mm}
\begin{picture}(15,5)
\put(2.85,1.45){$\bullet$}\put(1.85,2.45){$\bullet$}
\put(-1,0.1){\line(1,0){7}}\put(0,-.9){\line(0,1){5.5}}
\put(3,1.6){\line(-1,1){1}}\put(3.02,1.6){\line(-1,1){1}}
\put(3,1.6){\line(3,-1){1}}\put(3.02,1.6){\line(3,-1){1}}
\put(2,2.6){\line(-1,2){1}}\put(2.02,2.6){\line(-1,2){1}}
\put(5.4,.4){\line(-2,1){6.4}}\put(5,.8){\begin{footnotesize}$L_{\lambda}$\end{footnotesize}}
\put(13.85,1.45){$\bullet$}\put(11.85,2.45){$\bullet$}
\put(9,0.1){\line(1,0){7}}\put(10,-.9){\line(0,1){5.5}}
\put(14,1.6){\line(-2,1){2}}\put(14.02,1.6){\line(-2,1){2}}
\put(14,1.6){\line(3,-1){1}}\put(14.02,1.6){\line(3,-1){1}}
\put(12,2.6){\line(-1,2){1}}\put(12.02,2.6){\line(-1,2){1}}
\put(16,.6){\line(-2,1){7}}\put(15.6,1){\begin{footnotesize}$L_{\lambda}$\end{footnotesize}}
\put(13,2.2){\begin{footnotesize}$S$\end{footnotesize}}
\end{picture}
\end{center}
\end{figure}
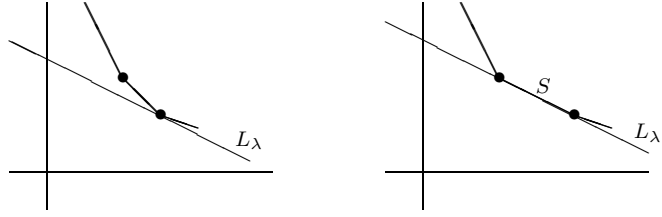

\noindent{\tt Routine ResidualPolynomial($\ty$,\,$\lambda$,\,$g$)}\vskip 1mm

\noindent INPUT

\noindent $-$ A type $\ty$ over $A$, of order $i-1\ge 0$, and a representative $\phi\in A[x]$ of $\ty$.

\noindent $-$ A slope $\lambda=-h/e\in\Q_{<0}$, with $h,e$ positive coprime integers.

\noindent $-$ A non-zero polynomial $g(x)\in K[x]$.\medskip

\noindent Let $g(x)=\sum_{0\le s}a_s(x)\phi(x)^s$ be the canonical $\phi$-adic expansion of $g(x)$. Let $S$ be the $\lambda$-component of $N_i(g)$, and let $s_0$ be the abscissa of the left end point of $S$. Let $d:=d(S)$ be the degree of $S$, so that $s_0+de$ is the rigth end point of $S$. 
The points of integer coordinates lying on $S$ have abscissa  $s_j:=s_0+je$, $0\le j\le d$. 

Compute, for each abscissa $s_j$, the residual coefficient $c_j\in\F_i$ defined as:
$$
c_j:=\begin{cases}
0,&\mbox{ if $(s_j,v_i(a_{s_j}\phi^{s_j})$ lies above }S,\\
z_{i-1}^{t_{i-1}(s_j)}R_{i-1}(a_{s_j})(z_{i-1}),&\mbox{ if $(s_j,v_i(a_{s_j}\phi^{s_j})$ lies on }S,             
\end{cases}
$$
where $t_0(s_j):=0$, $t_{i-1}(s_j)$ is described in \cite[Def. 2.19]{HN} for $i>1$, and $z_{i-1}\in\F_i$ is the image of $y$ through the isomorphism $\F_i\simeq \F_{i-1}[y]/(\psi_{i-1}(y))$.\medskip

\noindent OUTPUT:

\noindent $-$ The residual polynomial $R_{\lambda,i}(g)(y):=c_0+c_1y+\cdots+c_dy^d\in \F_i[y]$, with respect to the triple $(\ty,\phi,\lambda)$.\bigskip

The routine {\tt Construct} carries out the procedure described in \cite[Prop. 2.10]{HN}. It will only be used to construct representatives of the types. \bigskip  

\medskip

\noindent{\tt Routine Construct($\ty$,\,$\lambda$,\,$\varphi$,\,$V$)}\vskip 1mm

\noindent INPUT:

\noindent $-$ A type $\ty$ over $A$, of order $i-1\ge 0$, and a representative $\phi\in A[x]$ of $\ty$.

\noindent $-$ A slope $\lambda=-h/e\in\Q_{<0}$, with $h,e$ positive coprime integers.

\noindent $-$ A polynomial $\varphi(y)\in \F_i[y]$, of degree $d$.

\noindent $-$ A positive integer $V\ge ed(eV_i+h)$.\medskip

\noindent Let $(s,u)$ be minimal non-negative integers such that $V=ue+sh$. Our aim is to construct a polynomial $g(x)\in A[x]$, whose $i$-th order Newton polygon with respect to $(\ty,\phi)$ is contained in the segment of slope $\lambda$, degree $d$ and left end point $(s,u)$ (see Figure \ref{figConstruct}), and having moreover a prescribed residual polynomial. 

Let $\varphi(y)=a_0+a_1y+\dots+a_dy^d\in\F_i[y]$. If $i=1$, the coefficients $a_j\in\F_1=\F[y]/(\psi_0(y))$ can be expressed as polynomials in $z_0$ of degree less than $f_0$, with coefficients in $\F$. If we denote by $a_j(x)$ their arbitrary liftings to $A[x]$, we take:
$$
g(x)=\phi(x)^s\left(a_0(x)\pi^u+a_1(x)\pi^{u-h}\phi(x)^{e}+\cdots+a_d(x)\pi^{u-dh}\phi(x)^{de}\right).
$$

\noindent If $i>1$, the polynomial we are looking for is:
$$
g(x)=\phi(x)^s\left(g_0(x)+g_1(x)\phi(x)^{e}+\cdots+g_d(x)\phi(x)^{de}\right),
$$
where $g_j(x)\in A[x]$ are the output of {\tt Construct($\op{Trunc}_{i-1}(\ty)$,\,$\lambda_{i-1}$,\,$\varphi_j$,\,$w_j$)}, for adequate polynomials $\varphi_j(y)\in\F_{i-1}[y]$ with $\deg \varphi_j<f_{i-1}$, and integers $w_j\ge V_i$. \medskip

\noindent OUTPUT:

\noindent $-$ A polynomial $g(x)\in A[x]$ such that 
$v_{i+1}(g)=V$ and $y^{\ord_y(\varphi)}R_{\lambda,i}(g)(y)=\varphi(y)$.\medskip

\begin{figure}\caption{Routine \mbox{\tt Construct}.}\label{figConstruct}
\setlength{\unitlength}{5mm}
\begin{center}
\begin{picture}(10,6)
\put(2.85,3.55){$\bullet$}\put(6.85,1.5){$\bullet$}
\put(0,0.2){\line(1,0){10}}
\put(9,.7){\line(-2,1){8}}
\put(.9,4.7){\line(1,0){.2}}
\put(1,-.8){\line(0,1){6.5}}
\put(2.9,-.4){\begin{footnotesize}$s$\end{footnotesize}}
\put(6.3,-.4){\begin{footnotesize}$s+de$\end{footnotesize}}
\put(0.3,3.6){\begin{footnotesize}$u$\end{footnotesize}}
\put(-1.1,1.6){\begin{footnotesize}$u-dh$\end{footnotesize}}
\put(4.9,3){\begin{footnotesize}$N_i(g)$\end{footnotesize}}
\put(8.5,1.1){\begin{footnotesize}$\lambda$\end{footnotesize}}
\put(-.3,4.55){\begin{footnotesize}$V/e$\end{footnotesize}}
\multiput(3,.1)(0,.25){15}{\vrule height2pt}
\multiput(7,.1)(0,.25){7}{\vrule height2pt}
\multiput(.9,1.75)(.25,0){25}{\hbox to 2pt{\hrulefill }}
\multiput(.9,3.75)(.25,0){9}{\hbox to 2pt{\hrulefill }}
\end{picture}
\end{center}
\end{figure}

\noindent{\tt Routine Representative($\ty$)}\vskip 1mm

\noindent INPUT:

\noindent $-$ A type $\ty$ over $A$, of order $i\ge 1$.\medskip

\noindent Express $\psi_i(y)=y^{f_i}+\varphi(y)\in \F_i[y]$, for some polynomial $\varphi(y)$ of degree less than $f_i$. Let $g(x)$ be the output of {\tt Construct($\ty$,\,$\lambda_i$,\,$c\varphi$,\,$V_{i+1}$)}, for an adequate constant $c\in\F_i$ (cf. the proof of \cite[Thm. 2.11]{HN}).\medskip

\noindent OUTPUT:

\noindent $-$ A representative of $\ty$, constructed as: $\phi(x)=\phi_{i}(x)^{e_{i}f_{i}}+g(x)$.\bigskip
 
We now describe the Montes algorithm in pseudocode. Our design is slightly different from the original one. The output OM 
 representations are optimal and complete types of order $r+1$, as described in (\ref{OM}), where $r$ is the Okutsu depth of the corresponding $\p$-adic irreducible factor. In the original version, types of order $r+2$ were used in some ocasions (cf. \cite[Thm. 4.2]{okutsu}). The changes we introduce do not affect the complexity.   
The order of a type $\ty$ is the largest level $i$ for which all three fundamental invariants $(\phi_i,\lambda_i,\psi_i)$ are assigned. 

\noindent{\bf MONTES' ALGORITHM}\vskip 1mm

\noindent INPUT:

$-$ A monic separable polynomial $F(x)\in A[x]$.

$-$ A non-zero prime ideal $\p$ of $A$.\medskip

\st{1}Initialize an empty list {\tt OMReps}. 

\st{2}{\tt Factorization($\F$,$\overline{F}$)}.

\st{3}FOR each monic irreducible factor $\varphi$ of $\overline{F}$ DO

\stst{4}Take a monic lift, $\phi(x)\in A[x]$, of $\varphi$ and create a type $\ty$ of order zero with 
\vskip -.4mm \stst{}$\psi_0^\ty\leftarrow\varphi, \quad \omega_1^\ty\leftarrow \ord_\varphi \overline{F},\quad \phi_1^\ty\leftarrow \phi$.

\stst{5}Initialize an empty list {\tt Leaves}, and the list  {\tt Stack\,=[$\ty$]}. \vskip.1mm

\stst{}{\bf WHILE $\#${\tt Stack} $>0$ DO}

\ststst{6}Extract (and delete) the last  type $\ty_0$ from {\tt Stack}. Let $i-1$ be its order.

\ststst{7}{\tt Newton($\ty_0$,$\omega_i^{\ty_0}$,$F$)}. Let $N$ be the Newton polygon. 

\ststst{8}FOR every side $S$ of $N$ DO

\stststst{9}Set $\lambda_i^{\ty_0}\leftarrow$ slope of $S$. IF $\lambda_i^{\ty_0}=-\infty$, THEN add $\ty:=(\ty_0;(\phi_i^{\ty_0},-\infty,\hbox{--}))$
\vskip -.4mm \stststst{}to {\tt Leaves} and continue to the next side $S$. 

\stststst{10}{\tt ResidualPolynomial($\ty_0$,$\lambda_i^{\ty_0}$,$F$)}.

\stststst{11}{\tt Factorization($\F_i$,$R_i(F)$)}. 

\stststst{12}FOR every monic irreducible factor $\psi$ of $R_i(F)$ DO

\ststststst{13}Set $\ty\leftarrow\ty_0$, and extend $\ty$ to an order $i$ type
by setting  $\psi_i^{\ty}\leftarrow \psi$.

\ststststst{14}IF $\omega_i^{\ty_0}=1$, THEN add $\ty$ to {\tt Leaves} and go to {\bf 6}.  

\ststststst{15}Set $\om_{i+1}^\ty\leftarrow \ord_{\psi}R_i(F)$, and call {\tt Representative($\ty$)} to fill $\phi_{i+1}^\ty$.

\ststststst{16}IF $\deg \phi_{i+1}^\ty=\deg \phi_{i}^\ty$ THEN set 
$\phi_i^{\ty}\leftarrow \phi_{i+1}^{\ty}, \ \om_i^{\ty}\leftarrow\om_{i+1}^\ty$, and
\vskip -.4mm \ststststst{}delete all data in the $(i+1)$-th level of $\ty$. 

\ststststst{17}Add $\ty$ to {\tt Stack}.

\stst{}{\bf END WHILE}\vskip.1mm

\stst{18}Add all elements of {\tt Leaves} to the list {\tt OMReps}.\medskip

\noindent OUTPUT:\vskip .1mm

$-$ An OM factorization of $F$ over $\oo_\p[x]$, and the corresponding family  $\ty_{F_1},\dots,\ty_{F_g}$ of OM representations of the irreducible factors of $F$. The Okutsu factors are the $\phi$-polynomials at the last level of these types.\medskip

When the WHILE loop (corresponding to some irreducible factor $\varphi$ of $\overline{F}$) ends, the list {\tt Leaves} contains a tree of $F$-complete optimal types in 1-1 correspondence with all irreducible factors of $F(x)$ over $\oo_\p[x]$, which are congruent to a power of $\varphi$ modulo $\p$. The nodes of this tree (except for the root node) are labelled with a triple of fundamental invariants $(\phi_i,\lambda_i,\psi_i)$. Each leaf of the tree determines the type obtained by gathering the invariants of all nodes in the unique path joining the leaf to the root node. See Figure \ref{figForest}.  
 
\begin{figure}\caption{Connected tree of OM representations of the irreducible factors of $F$ whose reduction modulo $\p$ is a power of $\psi_0$. The leaves are represented by $\blacktriangle$.}
 \label{figForest}
\begin{center}
\setlength{\unitlength}{5.mm}
\begin{picture}(20,3.4)
\put(-.15,.6){$\bullet$}\put(1.85,1.65){$\bullet$}
\put(1.85,-.35){$\bullet$}
\put(0,.8){\line(2,1){2}}\put(0.02,.8){\line(2,1){2}}
\put(0,.8){\line(2,-1){2}}\put(0.02,.8){\line(2,-1){2}}
\put(-1,.7){\begin{footnotesize}$\psi_0$\end{footnotesize}}
\put(3.85,2.65){$\bullet$}
\put(3.85,.65){$\bullet$}
\put(2,1.8){\line(2,1){2}}\put(0.02,.8){\line(2,1){2}}
\put(2,1.8){\line(2,-1){2}}\put(0.02,.8){\line(2,-1){2}}
\put(2,-.2){\line(2,0){2}}\put(0.02,.8){\line(2,-1){2}}
\put(3.85,-.35){$\bullet$}
\put(4.5,.6){\begin{footnotesize}$\cdots$\end{footnotesize}}
\put(4.5,-.4){\begin{footnotesize}$\cdots$\end{footnotesize}}
\put(4.5,2.6){\begin{footnotesize}$\cdots$\end{footnotesize}}
\put(9.85,.65){$\bullet$}\put(11.85,1.65){$\bullet$}
\put(11.85,-.35){$\blacktriangle$}\put(11.85,.65){$\bullet$}
\put(8.6,.8){\line(1,0){3.4}}
\put(10,.8){\line(2,1){2}}\put(10.02,.8){\line(2,1){2}}
\put(10,.8){\line(2,-1){2}}\put(10.02,.8){\line(2,-1){2}}
\put(5.8,1.2){\begin{footnotesize}$(\phi_{i-1},\lambda_{i-1},\psi_{i-1})$\end{footnotesize}}
\put(10.8,2.2){\begin{footnotesize}$(\phi_i,\lambda_i,\psi_i)$\end{footnotesize}}
\put(12.6,-.4){\begin{footnotesize}$ F_{s}\leftrightsquigarrow\ty_{s} $\end{footnotesize}}
\put(12.7,.6){\begin{footnotesize}$\cdots$\end{footnotesize}}
\put(12.7,1.6){\begin{footnotesize}$\cdots$\end{footnotesize}}
\put(18.6,.6){\begin{footnotesize}$F_{t}\leftrightsquigarrow\ty_{t} $\end{footnotesize}}
\put(14.6,.8){\line(1,0){3.4}}
\put(15.85,.65){$\bullet$}\put(17.85,.65){$\blacktriangle$}
\end{picture}
\end{center}
\end{figure}
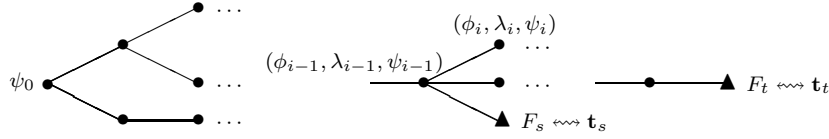

Step {\bf 16} takes care of the optimization. The list {\tt Stack} stores only strongly optimal types. If the enlarged type $\ty$ of order $i$ of step {\bf 13} has still this property, then it is sent to {\tt Stack}. Otherwise, we send again to {\tt Stack} the $(i-1)$-th order type $\ty_0$, but equipped with a different (and better) representative; this is called a \emph{refinement step} \cite[Sec. 3.2]{algorithm}. 

When the algorithm ends, the list {\tt OMReps} contains a forest (disjoint union of trees) of optimal $F$-complete types. Nevertheless, the list {\tt OMReps} is only a sequence of the leaves of all these trees, and the tree structure is not preserved.

\subsection*{The Montes algorithm as an irreducibility test}
For any level $i$, the existence of two sides of different slope in $N_i^-(F)$, or two coprime factors of $R_i(F)$ in $\F_i[y]$, implies that $F(x)$ is not irreducible \cite[Thms. 3.1,3.7]{HN}. On the other hand, if no factorization has been detected in lower levels, the Newton polygon $N_i^-(F)$ is one-sided and the corresponding residual polynomial $R_i(F)$ is irreducible in $F_i[y]$, then $F(x)$ is irreducible \cite[Cor. 3.8]{HN}.     

Therefore, we can use the following version of the 
Montes algorithm as an irreducibility test for polynomials over $\oo_\p[x]$. \medskip

\noindent{\bf IRREDUCIBILITY TEST}\vskip 1mm

\noindent INPUT:

$-$ A monic separable polynomial $F(x)\in A[x]$.

$-$ A non-zero prime ideal $\p$ of $A$.\medskip

\st{1}{\tt Factorization($\F$,$\overline{F}$)}.
IF there are at least two irreducible factors THEN
\vskip -.4mm \st{}return {\tt  false}.

\st{2}Consider a monic lift, $\phi(x)\in A[x]$, of the unique irreducible factor $\varphi$ of $\overline{F}$ and
\vskip -.4mm \st{}create a type $\ty$ of order zero with: \quad
$\psi_0^\ty\leftarrow\varphi, \quad \omega_1^\ty\leftarrow \ord_\varphi\overline{F},\quad\phi_1^\ty\leftarrow \phi$.

\st{3}Initialize the list  {\tt Stack\,=[$\ty$]}.\vskip .1mm

\st{}{\bf WHILE $\#${\tt Stack} $>0$ DO}

\stst{4}Extract (and delete) the last  type $\ty_0$ from {\tt Stack}. Let $i-1$ be its order.

\stst{5}{\tt $N$=Newton($\ty_0$,$\omega_i^{\ty_0}$,$F$)}.
IF $N$ has at least two sides THEN return {\tt  false}.

\stst{6}Set $\lambda_i^{\ty_0}\leftarrow$ slope of the unique side of $N$. IF $\lambda_i^{\ty_0}=-\infty$, THEN return {\tt true}.

\stst{7}{\tt ResidualPolynomial($\ty_0$,$\lambda_i^{\ty_0}$,$F$)}.

\stst{8}{\tt Factorization($\F_i$,$R_i(F)$)}. 
IF there are at least two irreducible factors
\vskip -.4mm \stst{}THEN return {\tt  false}, ELSE
let $\psi$ be the unique irreducible factor of $R_i(F)$.

\stst{9}Set $\ty\leftarrow\ty_0$, and extend $\ty$ to an order $i$ type by setting
$\psi_i^{\ty}\leftarrow \psi$.

\stst{10}IF $\ord_{\psi}R_i(F)=1$, THEN return {\tt  true}.  

\stst{11}Set $\om_{i+1}^\ty\leftarrow \ord_{\psi}R_i(F)$, and call {\tt Representative($\ty$)} to fill $\phi_{i+1}^\ty$.

\stst{12}IF $\deg \phi_{i+1}^\ty=\deg \phi_{i}^\ty$ THEN set 
$\phi_i^{\ty}\leftarrow \phi_{i+1}^{\ty}, \ \om_i^{\ty}\leftarrow\om_{i+1}^\ty$, and delete all data
\vskip -.4mm \stst{}in the $(i+1)$-th level of $\ty$. 

\stst{13}Add $\ty$ to {\tt Stack}.

\st{}{\bf END WHILE}\vskip .1mm

\noindent OUTPUT:\vskip .1mm

{\tt true} if $F(x)$ is irreducible over $\oo_\p[x]$ and {\tt false} otherwise.

\section{Complexity analysis of the Montes algorithm}\label{secComplexity}
All tasks we are interested in may be performed modulo $\p^\nu$, for a sufficiently high precision $\nu$. Thus, we may assume that the elements of $A$ are finite $\pi$-adic developments. In particular, the computation of the $\p$-adic valuation $v=v_\p$ has a negligible cost. 

\begin{definition}
An operation of $A$ is called \emph{$\p$-small} if it involves two elements belonging to a fixed system of representatives of $A/\p$.
\end{definition}

Working at precision $\nu$, each multiplication in $A$ costs $O(\nu^{1+\epsilon})$
$\p$-small operations if we assume the fast multiplications techniques of 
Sch\"onhage-Strassen \cite{SS}.

Let $q:=\#\F$. We assume that a $\p$-small operation is equivalent to  $O\left(\log(q)^{1+\epsilon}\right)$ word operations, the cost of an operation in the residue field $\F=A/\p$. This is the case in most of the Dedekind rings that naturally arise in practice.

\subsection{Complexity of the basic subroutines}\label{subsecSubroutines} 

\begin{lemma}{\cite[Cor. 14.30]{vzGG}}\label{Factorization}
Let $\mathcal{F}$ be a finite field with $q_{\mathcal{F}}$ elements, and $g(x)\in \mathcal{F}[x]$ a polynomial of degree $d$. The cost of the routine {\tt Factorization($\mathcal{F}$,\,$g$)} is $O\left(d^{2+\epsilon}+d^{1+\epsilon}\log(q_{\mathcal{F}})\right)$ operations in $\mathcal{F}$. 
\end{lemma}

The following observation is easy to prove by an inductive argument.

\begin{lemma}\label{trivial}
Let $m_1,\dots,m_i$ be positive integers such that $m_1\mid \cdots \mid m_i$ and $m_1<\cdots<m_i$. Then, $m_1+\cdots +m_i\le 2m_i$. 
\end{lemma}

\begin{lemma}\cite[Lem. 18]{pauli}\label{phiadic}
Let $\ty$ be a strongly optimal type of order $i-1\ge1$. 
Let $a(x)\in\oo[x]$ be a polynomial of degree less than $m_{i}$. The computation of the multiadic expansion of $a(x)$,
\begin{equation}\label{multiadic}
a(x)=\sum_{\j=(j_1,\dots,j_{i-1})} a_\j(x)\phi_1(x)^{j_1}\cdots \phi_{i-1}(x)^{j_{i-1}},\quad \deg a_\j<m_1,
\end{equation}
where $0\le j_k<e_k f_k$, for all $1\le k< i$, has a cost of $O((m_{i})^{1+\epsilon})$ operations in $A$.
\end{lemma}

Actually, in \cite{pauli} it was proved an estimation of $O(m_i^2)$ operations in $A$, assuming ordinary arithmetic. If we assume fast multiplication, the cost of the computation of the $\phi_{i-1}$-expansion of $a(x)$ may be estimated in $O((m_{i})^{1+\epsilon})$ operations in $A$ \cite[Thm. 9.15]{vzGG}. By using this estimation, the proof of \cite[Lem. 18]{pauli} leads to Lemma \ref{phiadic}.

\begin{lemma}\label{Newton}
Let $\ty$ be a strongly optimal type of order $i-1\ge0$, with representative $\phi(x)$. Let $\omega$ be a positive integer and $g(x)\in A[x]$ a polynomial of degree $d\ge \omega m_i$. Then, the cost of the routine {\tt Newton($\ty$,$\omega$,$g$)} is $O(\omega d^{1+\epsilon})$ operations in $A$.
\end{lemma}

\begin{proof}
The computation of the first $\omega+1$ coefficients of the $\phi$-development of $g(x)$ requires $\omega+1$ divisions with remainder:
$$
g=\phi\cdot q_1+a_0,\quad
q_1=\phi\cdot q_2+a_1,\quad
\cdots\quad,\quad 
q_\omega=\phi\cdot q_{\omega+1}+a_\omega.
$$
The number of operations in $A$ that are necessary to carry out each one of this divisions is $O(d^{1+\epsilon})$ \cite[Thm. 9.6]{vzGG}. Thus, the lemma states that the cost of this initial step dominates the whole routine.
 
The next step is the computation of $v_i(a_k)$, for $0\le k\le \omega$. Denote by $a(x)=a_k(x)$ any one of these $\omega+1$ coefficients, and consider the multiadic development (\ref{multiadic}) of $a(x)$.
By \cite[Lem. 4.2]{bases}, we have 
\begin{equation}\label{v1}
v_i(a(x))=\min_{\j=(j_1,\dots,j_{i-1})}\{v_i(a_\j)+j_1v_i(\phi_1)+\cdots+j_{i-1}v_i(\phi_{i-1})\}.
\end{equation}
By \cite[Prop. 2.15]{HN}, we may use closed formulas for the values $v_i(\phi_j)$ in terms of the Okutsu invariants, and since $\deg a_\j<m_1$, \cite[Prop. 2.7]{HN} shows that:
$$
v_i(a_\j)=e_0\cdots e_{i-1}\min\{v_\p(c)\mid c\mbox{ coefficient of }a_\j(x)\}. 
$$
Thus, the cost of computing $v_i(a_k)$ is dominated by the cost of the computation of the multiadic development of $a_k$. By Lemma \ref{phiadic}, the total cost of this step  is $(\omega+1)O\left((m_i)^{1+\epsilon}\right)$ operations in $A$.  This cost is clearly dominated by the cost of the first divisions with remainder.

Finally, the computation of the Newton polygon has a cost of $O((\omega+1)^2)$ multiplications of integers. If we work at precision $\nu$, (\ref{v1}) shows that $e_0\cdots e_{i-1}\nu$ is an upper bound of $v_i(a_k)$; hence,  each multiplication of integers of this size requires $O(\log(m_i\nu)^{1+\epsilon})$ word operations. Since $\omega\le d/m_i$, the complexity of this task is also dominated by that of the first divisions with remainder, obtained by multiplying by $(\nu\log(q))^{1+\epsilon}$ the cost in number of operations in $A$.
\end{proof}

\begin{lemma}\label{ResidualPolynomial}
Let $\ty$ be a strongly optimal type of order $i-1\ge0$, with representative $\phi(x)$, and take $\lambda\in\Q_{<0}$, $g(x)\in A[x]$. Let $S$ be the $\lambda$-component of $N_i(g)$, and let $d=d(S)$ be the degree of $S$. Then, the cost of {\tt ResidualPolynomial($\ty$,$\lambda$,$g$)} is
$O\left(d(f_0\cdots f_{i-1})(m_i)^{1+\epsilon}\log(q)\right)$ $\p$-small operations. 
\end{lemma}

\begin{proof}Let $e$ be the least positive denominator of $\lambda$. Let $s_0$ be the abscissa of the left end point of $S$, and take $s_j:=s_0+je$, for $0\le j\le d$.
We assume that in a previous call to the routine {\tt Newton}, we computed (and stored) the coefficients $a_{s_j}$ of the $\phi$-adic expansion of $g(x)$, and their $(\phi_1,\dots,\phi_{i-1})$-multiadic expansion. Also, along this computation it is easy to store the necessary data to compute the exponents $t_{i-1}(s_j)$ at zero cost \cite[Def. 2.19]{HN}.

Thus, the computation of the coefficients $c_0,\dots,c_d\in \F_i$ of the residual polynomial $R_{\lambda,i}(g)(y)$, requires two tasks:
\begin{enumerate}
\item[(a)] compute $R_{i-1}(a_{s_j})(y)\in\F_{i-1}[y]$, for each $0\le j\le d$,

\item[(b)] compute $c_j:=z_{i-1}^{t_{i-1}(s_j)}R_{i-1}(a_{s_j})(z_{i-1})\in F_i$, for each $0\le j\le d$.
\end{enumerate}

Denote by $C_i(d)$ the cost of the computation of $R_{\lambda,i}(g)$, measured in number of $\p$-small operations. Since $\deg a_{s_j}<m_i=e_{i-1}f_{i-1}m_{i-1}$, the Newton polygon $N_{i-1}(a_{s_j})$ has length less than $e_{i-1}f_{i-1}$; hence, the $\lambda_{i-1}$-component of this polygon has degree less than $f_{i-1}$. Therefore, the cost of task (a) is dominated by $C_{i-1}(f_{i-1})$.

The computation of $z_{i-1}^{t_{i-1}(s_j)}$ requires $O\left(\log(\#\F_i)\right)$ multiplications in $\F_i$. Since $\#\F_i=q^{f_0\cdots f_{i-1}}$, the cost is
$O\left((f_0\cdots f_{i-1})^{2+\epsilon}\log(q)\right)$ $\p$-small operations. 

Since $\deg R_{i-1}(a_{s_j})<f_{i-1}$, the cost of the computation of $R_{i-1}(a_{s_j})(z_{i-1})$ by Horner's rule is $O(f_{i-1})$ multiplications in $\F_i$; thus, it is dominated by the computation of a power of $z_{i-1}$.
Altogether, we get 
$$C_i(d)\le (d+1)\left(C_{i-1}(f_{i-1})+(f_0\cdots f_{i-1})^{2+\epsilon}\log(q)\right).$$
From this recurrence, it is easy to derive:
$$
C_i(d)=(d+1)O\left(f_0\cdots f_{i-1}\log(q)\left(f_0^{1+\epsilon}+(f_0f_1)^{1+\epsilon}+\cdots+(f_0\cdots f_{i-1})^{1+\epsilon}\right)\right).
$$   

Finally, we may use Lemma \ref{trivial} to estimate:
$f_0^{1+\epsilon}+\cdots+(f_0\cdots f_{i-1})^{1+\epsilon}\le
\left(m_0\right)^{1+\epsilon}+\cdots+\left(m_i\right)^{1+\epsilon}=O\left(m_i^{1+\epsilon}\right)$.
\end{proof}

\begin{lemma}\label{Construct}
Let $\ty$ be a strongly optimal type of order $i-1\ge0$, with representative $\phi(x)$. Let $\lambda=-h/e$, where  $h,e$ are positive coprime integers.
Let $\varphi(y)\in \F_i[y]$ be a polynomial of degree $d$, and $V\ge ed(eV_{i}+h)$ a positive integer.
Then, the cost of {\tt Construct($\ty$,$\lambda$,$\varphi$,$V$)} is $O\left((f_0\cdots f_{i-1}d)^{2+\epsilon}V^{1+\epsilon}\right)$ $\p$-small operations.
\end{lemma}

\begin{proof}
The output polynomial is constructed as: 
$$
g(x)=\phi(x)^s\left(g_0(x)+g_1(x)\phi(x)^{e}+\cdots+g_d(x)\phi(x)^{de}\right),
$$
where $0\le s<e$, and the polynomials $g_j(x)\in A[x]$ may be taken as the output of an adequate call to {\tt Construct} at level $i-1$. In particular, $\deg g_j<m_i$, for all $j$.  

We must compute the polynomials $\phi(x)^s$, $\phi(x)^{e}$, $g_0(x),\dots,g_d(x)$, and finally compute $g(x)$ by Horner's rule. This latter task requires $d+1$ multiplications of polynomials. In each multiplication, the two factors have degree (bounded by) $$(m_i,em_i), \ ((e+1)m_i,em_i),\ ((2e+1)m_i,em_i),\ \dots,\ (((d+1)e+1)m_i,em_i),$$
respectively. The multiplication of two polynomials of degrees $m'\le m$ requires $O(m^{1+\epsilon})$ operations in $A$. Thus, if we denote $m_{i+1}:=edm_i$, the number of operations in $A$ required for the final evaluation of $g(x)$ is of the order of:
$$(em_i)^{1+\epsilon}(1^{1+\epsilon}+2^{1+\epsilon}+\cdots+d^{1+\epsilon})=O((em_i)^{1+\epsilon}d^{2+\epsilon})=O(d(m_{i+1})^{1+\epsilon}).
$$
This estimation clearly dominates the cost of the computation of $\phi(x)^s$ and $\phi(x)^{e}$. Thus, we analyze only the cost of the computation of $g_0(x),\dots,g_d(x)$. 

Denote by $C_i(d)$ the total cost of {\tt Construct}, measured in number of operations in $A$.
We have seen that $C_i(d)=d\left(C_{i-1}(f_{i-1})+O((m_{i+1})^{1+\epsilon})\right)$. By using Lemma \ref{trivial}, this recurrence leads to:
\begin{equation}\label{estfast}
\begin{aligned}
C_i(d)=&\;O\left(d\,(m_{i+1})^{1+\epsilon}+d\,f_{i-1}(m_i)^{1+\epsilon}+\cdots+d\,f_{i-1}\cdots f_0(m_0)^{1+\epsilon}\right)\\=&\;O\left(d\,f_{i-1}\cdots f_0\left((m_{i+1})^{1+\epsilon}+(m_i)^{1+\epsilon}+\cdots+(m_0)^{1+\epsilon}\right)\right)\\=&\;O\left(d\,f_{i-1}\cdots f_0\,(m_{i+1})^{1+\epsilon}\right).
\end{aligned} 
\end{equation}

Finally, we may work with precision $\nu:=\lfloor V/(e_0\cdots e_{i-1}e)\rfloor +1$, without changing the desired properties for $g(x)$:
$$v_{i+1}(g)=V,\quad y^{\ord_y\varphi}R_{\lambda,i}(g)(y)=\varphi(y),$$ 
where $v_{i+1}$ is the MacLane valuation determined by $\ty$, $\phi$ and $\lambda$. In fact, suppose  $G(x)=g(x)+h(x)$, for a polynomial $h(x)\in A[x]$, all whose coefficients $c$ satisfy $v_\p(c)>V/(e_0\cdots e_{i-1}e)$. Then, $v_{i+1}(c)=(e_0\cdots e_{i-1}e)v_\p(c)>V$, by Lemma \ref{previous}, so that $v_{i+1}(h)>v_{i+1}(g)$, and  $v_{i+1}(G)=v_{i+1}(g)$. Also, we get $R_{\lambda,i}(G)(y)=R_{\lambda,i}(g)(y)$ by \cite[Prop. 2.8]{HN}.  

Therefore, the total cost of {\tt Construct}, measured in number of $\p$-small operations, is obtained by multiplying the estimation of (\ref{estfast}) by $\nu^{1+\epsilon}$.
\end{proof}

\begin{corollary}\label{Representative}
Let $\ty'=(\psi_0;(\phi_1,\lambda_1,\psi_1);\cdots;(\phi_{i-1},\lambda_{i-1},\psi_{i-1});(\phi,\lambda,\psi))$ be an optimal type of order $i\ge1$, where
$\lambda=-h/e$ for some positive coprime integers $h,e$, and $y\ne\psi(y)\in\F_i[y]$ is a monic irreducible polynomial of degree $f$. Let $V:=ef(eV_i+h)$. The cost of the computation of a representative $\phi'$ of $\ty'$ is $O\left((f_0\cdots f_{i-1}f)^{2+\epsilon}V^{1+\epsilon}\right)$ $\p$-small operations.
\end{corollary}

\begin{proof}
The polynomial $\phi'(x)$ is constructed as $\phi(x)^{ef}+g(x)$, where $g(x)$ is the output of the routine {\tt Construct($\ty$,$\lambda$,$\psi(y)-y^{f}$,$V$)}. The computation of $\phi^{ef}$ by repeated squarings costs $O((efm_{i})^{1+\epsilon})$ operations in $A$; this cost is dominated by the estimation (\ref{estfast}) of the cost of the computation of $g(x)$.
Thus, the corollary is an immediate consequence of Lemma \ref{Construct}.  
\end{proof}

\subsection{Complexity of the polynomial irreducibility test}\label{subsecIrrTest} 

The aim of this section is to prove a new estimation for the complexity of the polynomial irreducibili\-ty test based on the Montes algorithm. In comparison with previous estimations \cite{FV,pauli}, the total degree in $n$ and $\delta$ is reduced from $4+\epsilon$ to $2+\epsilon$.

\begin{theorem}\label{IrrEstimation}
The cost of the irreducibility test over $\oo_\p[x]$, applied to a monic se\-parable polynomial $F\in A[x]$ of degree $n$ is $O\left(n^{2+\epsilon}+n^{1+\epsilon}(1+\delta)\log(q)+\delta^{2+\epsilon}\right)$ $\p$-small operations, where $\delta:=v_\p(\dsc(F))$.
\end{theorem}

\begin{corollary}\label{IrrEstimationSmall}
If we assume $\p$ small (i.e. $\log(q)=O(1)$), we obtain an estimation of $O(n^{2+\epsilon}+\delta^{2+\epsilon})$ word operations.
\end{corollary}

Before proving this theorem, we discuss some features of the flow of the algorithm.
The irreducibility test provides as a by-product an optimal type $\ty$ of order $r$, represented by a tree with unibranch nodes and a unique leaf:

\begin{center}
\setlength{\unitlength}{5.mm}
\begin{picture}(17,2.6)
\put(-.15,.85){$\bullet$}\put(1.85,.85){$\bullet$}\put(3.85,.85){$\bullet$}
\put(0,1){\line(1,0){5.8}}
\put(-1,1){\begin{footnotesize}$\psi_0$\end{footnotesize}}
\put(.8,1.5){\begin{footnotesize}$(\phi_1,\lambda_1,\psi_1)$\end{footnotesize}}
\put(3,0.2){\begin{footnotesize}$(\phi_2,\lambda_2,\psi_2)$\end{footnotesize}}
\put(6.2,.95){\begin{footnotesize}$\cdots\cdots\cdots\cdots$\end{footnotesize}}
\put(11.5,.85){$\bullet$}\put(13.5,.85){$\bullet$}
\put(9.8,1){\line(1,0){3.8}}
\put(8,1.5){\begin{footnotesize}$(\phi_{r-1},\lambda_{r-1},\psi_{r-1})$\end{footnotesize}}
\put(14,1){\begin{footnotesize}$(\phi_r,\lambda_r,\psi_r)\leftrightsquigarrow\ty $\end{footnotesize}}
\end{picture}
\end{center}\medskip

If $\ord_\ty(F)=1$, then $F$ was recognized to be irreducible. Otherwise, after eventually several refinement steps, a representative $\phi_{r+1}$ of $\ty$ was found, such that $N_{r+1}(F)$ had more than one side, or $R_{r+1}(F)$ had more than one irreducible factor; then, $F$ was recognized to be reducible. In this latter case, all irreducible factors of $F$ are of type $\ty$ (Definition \ref{type}), and they have degree a multiple of $m_{r+1}$, by Lemma \ref{oftype}. In particular, $n=fm_{r+1}$, for some integer $f\ge 2$. 

We may choose a monic irreducible polynomial $\psi\in\F_{r+1}[y]$ of degree $f$ and use {\tt Representative} to construct a representative $\phi\in A[x]$ of the type of order $r+1$:
$$\ty'=(\psi_0;(\phi_1,\lambda_1,\psi_1);\cdots;(\phi_r,\lambda_r,\psi_r);(\phi_{r+1},-1,\psi)). 
$$ 
The polynomial $\phi$ is irreducible over $\oo_\p$ and it has degree $m_{r+2}=fm_{r+1}=n$. The irreducibility test applied to $\phi$ performs the same steps at all levels $i\le r$, the same refinement steps at level $r+1$ to find $\phi_{r+1}$, and it will compute $N_{r+1}(\phi)$ and $R_{r+1}(\phi)$, to deduce the irreducibility
of $\phi$ from the property $R_{r+1}(\phi)\sim \psi$. We shall see below that the cost of reaching $\phi_{r+1}$ depends only on $n$ and $\phi_{r+1}$. By Lemmas \ref{Newton}, \ref{ResidualPolynomial}, the cost of the computation of $N_{r+1}(\phi)$, $R_{r+1}(\phi)$ is not lower than the cost of the computation of $N_{r+1}(F)$, $R_{r+1}(F)$, respectively. Hence:

\begin{remark}\label{irreducible}\rm
For the estimation of the complexity of the irreducibility test, we may assume that the input polynomial is irreducible.
\end{remark}

For the estimation of the complexity we need to estimate the cost of advancing from the $(i-1)$-th node of the tree to the $i$-th node. This step may require several iterations of the WHILE loop, because of the refinement steps
at the $i$-th level. Thus, the crucial questions are the evaluation of the cost of each iteration at the $i$-th level and to find an upper bound for the number of these iterations. 

\begin{lemma}\label{refinement}
The width of $F$ at the $i$-th level, $\lceil |\lambda_i|
\rceil$, is an upper bound for the number of iterations of the WHILE loop at the $i$-th level, that are necessary to reach the right values of  $(\phi_i,\lambda_i,\psi_i)$.
\end{lemma}

\begin{proof}
The first WHILE loop at the $i$-th level picks the type of order $i-1\ge0$: 
$$\ty=(\psi_0;(\phi_1,\lambda_1,\psi_1);\cdots;(\phi_{i-1},\lambda_{i-1},\psi_{i-1})),$$
and a representative $\phi$ of degree $m_i$ (a first candidate to be the polynomial $\phi_{i}$),
from the {\tt Stack}. Then, it computes the slope $\lambda=-h/e$, with $h,e$ positive coprime integers, of the one-sided Newton polygon $N_i(F)$ with respect to $(\ty,\phi)$, and the unique irreducible factor $\psi\in\F_i[y]$ of the residual polynomial $R_{\lambda,i}(F)$. Finally, it constructs a representative $\phi'$ of the type
$$\ty'=(\psi_0;(\phi_1,\lambda_1,\psi_1);\cdots;(\phi_{i-1},\lambda_{i-1},\psi_{i-1});(\phi,\lambda,\psi)).$$
Let $f=\deg\psi$, $V=ef(eV_i+h)$. By \cite[Thms. 2.11,3.1]{HN}, $\deg \phi'=efm_i$, and
$$ v(\phi(\t))=(V_i+|\lambda|)/(e_0\cdots e_{i-1})< v(\phi'(\t))=(V+|\lambda'|)/(e_0\cdots e_{i-1}e),
$$ 
where $\t$ is a root of $F$ in $\overline{K}_\p$, $v$ is the canonical extension of $v_\p$ to  $\overline{K}_\p$, and $\lambda'$ is the slope of the Newton polygon $N_{i+1}(F)$, computed with respect to $(\ty',\phi')$.

The loop is a refinement step if and only if $\deg\phi'=m_i$, or equivalently, $e=f=1$. In this case, $\phi'$ is also a representative of $\ty$, and we proceed to a new iteration of the WHILE loop at the $i$-th level, with the pair $(\ty,\phi')$ as starting data. Otherwise, \cite[Thm. 3.1]{algorithm} shows that $v(\phi(\t))$ is maximal among all other representatives of $\ty$; thus, it may be taken as an Okutsu polynomial of the $i$-th level. We take $\phi_i:=\phi$, $\lambda_i:=\lambda$, $\psi_i:=\psi$ and we proceed to a new iteration of the WHILE loop at the $(i+1)$-th level with the pair $(\ty',\phi')$ as starting data. 

Therefore, the number of iterations of the WHILE loop at the $i$-th level is bounded from above by the number of values of $v(\phi(\t))$, where $\phi$ runs on all possible representatives of $\ty$. This number of values is $\lceil|\lambda_i|\rceil$ by Proposition \ref{compwidth}.
\end{proof}

\noindent{\sl Proof of Theorem \ref{IrrEstimation}. }By Remark \ref{irreducible} we may assume that the input polynomial $F$ is irreducible over $\oo_\p[x]$. Let $r$ be the Okutsu depth of $F$ and $\ty_{F,r}=(\psi_0;(\phi_1,\lambda_1,\psi_1);\cdots;(\phi_{r},\lambda_{r},\psi_{r}))$ the strongly optimal type of order $r$ computed along the flow of the algorithm.
 
We shall frequently use an estimation that is an immediate consequence of formula (\ref{delta0}) and the inequality $\delta_0(F)\le 2\delta/n$ of Lemma \ref{delta0props}:
\begin{equation}\label{sumdelta}
 \sum\nolimits_{1\le i\le r}\dfrac{|\lambda_i|}{e_0\cdots e_{i-1}}\,\dfrac{n^2}{m_i}=O(\delta).
\end{equation}
 
The initial steps compute the pair $(\ty,\phi)$, where $\ty=(\psi_0)$ is the type of order zero determined by the unique irreducible factor of $F$ modulo $\p$ and $\phi$ is a monic lift to $A[x]$ of $\psi_0$. The cost of these operations is dominated by the factorization of $F$ modulo $\p$, which costs $O\left(n^{2+\epsilon}+n^{1+\epsilon}\log(q)\right)$ $\p$-small operations. 

Each iteration of the WHILE loop calls only once each subroutine {\tt Newton}, {\tt ResidualPolynomial}, {\tt Factorization} and {\tt Representative}. Let {\tt R} be one of these subroutines; by Lemma \ref{refinement}, the total cost of the calls to {\tt R} of all iterations of the WHILE loop is:
\begin{equation}\label{sum}
\sum\nolimits_{1\le i\le r}|\lambda_i|\,C_{\mbox{\tt R},i},
\end{equation}
where $C_{\mbox{\tt R},i}$ is an upper bound of the cost of any call to {\tt R} along the different iterations of the WHILE loop at the $i$-th level. We proceed to estimate 
$C_{\mbox{\tt R},i}$ and (\ref{sum}), for each subroutine. We keep the notation introduced in the proof of Lemma \ref{refinement} for the data $\ty$, $\phi$, $\lambda$, $\psi$, $e$, $f$, $h$, $\phi'$, $V$, $\ty'$, used in any of these iterations.\medskip

\noindent{\tt R=Newton. }By Lemma \ref{Newton}, the cost of one call to {\tt Newton}  depends only on $n=\deg F$ and $\omega:=\ell(N_i(F))=n/\deg\phi$. Since $\deg\phi=m_i=e_{i-1}f_{i-1}m_{i-1}$ does not depend on the choice of $\phi$, the cost  is constant for all the iterations at the $i$-th level. By Lemma \ref{Newton}, this cost is $O\left((n/m_i)n^{1+\epsilon}\right)$ operations in $A$. 

By Theorem \ref{bound}, we may work at any precision $\nu>2\delta/n$, so that we may take 
$$
C_{\mbox{\tt R},i}=O\left((n/m_i)n^{1+\epsilon}(\delta/n)^{1+\epsilon}\right)=
O\left((n/m_i)\delta^{1+\epsilon}\right)
$$
$\p$-small operations. By (\ref{sumdelta}), we get:
$$
\sum_{1\le i\le r}|\lambda_i|\,C_{\mbox{\tt R},i}=
\delta^{1+\epsilon}\sum_{1\le i\le r}|\lambda_i|\,\dfrac{n}{m_i}\le 
\delta^{1+\epsilon}\sum_{1\le i\le r}\dfrac{|\lambda_i|}{e_0\cdots e_{i-1}}\,\dfrac{n^2}{m_i}=O\left(\delta^{2+\epsilon}\right).
$$

\noindent{\tt R=ResidualPolynomial. }By Lemma \ref{ResidualPolynomial}, the cost of one call to {\tt ResidualPolyno\-mial}  depends only on $f_0,\dots,f_{i-1}$ and the degree of the side $d(N_i(F))=\omega/e=n/(m_ie)$. Thus, the cost is constant for all refinement steps ($e=1$) and eventually lower in the last iteration of WHILE ($ef>1$). By Lemma \ref{ResidualPolynomial},  we may take 
\begin{equation}\label{Crespol}
C_{\mbox{\tt R},i}=O\left((n/m_i)(f_0\cdots f_{i-1})m_i^{1+\epsilon}\log(q)\right)=O\left((n/e_0\cdots e_{i-1})n^{1+\epsilon}\log(q)\right)
\end{equation}
$\p$-small operations. By (\ref{sumdelta}), we get:
$$
\sum\nolimits_{1\le i\le r}|\lambda_i|\,C_{\mbox{\tt R},i}\le
n^{1+\epsilon}\log(q)\sum\nolimits_{1\le i\le r}\dfrac{|\lambda_i|\,n}{e_0\cdots e_{i-1}}=O\left(n^{1+\epsilon}\log(q)\delta\right).
$$

\noindent{\tt R=Factorization. }By Lemma \ref{Factorization}, the cost of one call to {\tt Factorization}  depends only on $\deg R_{\lambda,i}(F)=d(N_i(F))=\omega/e=n/(m_ie)$, and it is bounded from above by $O\left((n/m_i)^{2+\epsilon}+(n/m_i)^{1+\epsilon}(f_0\cdots f_{i-1})\log(q)\right)$ operations in $\F_i$. We may estimate 
\begin{align*}
C_{\mbox{\tt R},i}&\ =O\left((n/m_i)^{2+\epsilon}(f_0\cdots f_{i-1})^{1+\epsilon}+(n/m_i)^{1+\epsilon}(f_0\cdots f_{i-1})^{2+\epsilon}\log(q)\right)\\
&\ =O\left(n^{2+\epsilon}/(m_i(e_0\cdots e_{i-1})^{1+\epsilon})+(n/e_0\cdots e_{i-1})^{1+\epsilon}f_0\cdots f_{i-1}\log(q)\right)
\end{align*}
$\p$-small operations. Both summands of this expression are dominated by the estimation of (\ref{Crespol}). Thus, the total cost of {\tt Factorization} is dominated by the total cost of {\tt ResidualPolynomial}.\medskip

\noindent{\tt R=Representative. }By Corollary \ref{Representative}, the cost of one call to {\tt Representative} is $O\left((f_0\cdots f_{i-1}f)^{2+\epsilon}V^{1+\epsilon}\right)$ $\p$-small operations. Along the refinement steps, we have $f=1$, $V=V_i+h$; since the value of $h=|\lambda|$ grows at each iteration, the cost is dominated by the cost of the last iteration, where $f=f_i$, $V=V_{i+1}=e_if_i(e_iV_i+h_i)$. Thus, we may take 
$$
C_{\mbox{\tt R},i}=O\left((f_0\cdots f_{i-1}f_i)^{2+\epsilon}(V_{i+1})^{1+\epsilon}\right)\mbox{ $\p$-small operations}.
$$
By the recurrent formulas for $V_i$ in section \ref{secOkutsu}, and Lemma \ref{delta0props}, $V_{i+1}/(e_0\cdots e_i)\le V_{r+1}/(e_0\cdots e_r)\le 2\delta/n$. Hence, $f_0\cdots f_iV_{i+1}\le 2\delta$.
By (\ref{sumdelta}), we get:
\begin{align*}
\sum_{1\le i\le r}|\lambda_i|\,C_{\mbox{\tt R},i}\le&\ (2\delta)^{1+\epsilon}\sum_{1\le i\le r}|\lambda_i|\,f_0\cdots f_i\,=\, (2\delta)^{1+\epsilon}\sum_{1\le i\le r}\dfrac{|\lambda_i|}{e_0\cdots e_{i-1}}\,m_if_i\\\le&\ (2\delta)^{1+\epsilon}\sum\nolimits_{1\le i\le r}\dfrac{|\lambda_i|}{e_0\cdots e_{i-1}}\,\dfrac{n^2}{m_i}=O\left(\delta^{2+\epsilon}\right).
\end{align*}
This completes the proof of the theorem.
\qed

\subsection{Complexity of the general factorization algorithm}\label{subsecGeneral}
\mbox{\null}\medskip

Let $F_1,\dots,F_g\in\oo_\p[x]$ be the monic irreducible factors of the input polynomial $F\in A[x]$. Denote $n_s=\deg F_s$, $\delta_s=\delta(F_s)$, and let $r_s$ be the Okutsu depth of $F_s$, for all $1\le s\le g$. 

The output of the Montes algorithm is a forest $\tcal=\tcal_1\cup \cdots \cup \tcal_k$, disjoint union of $k$ connected trees, one for each irreducible factor of $\overline{F}$. 
Let $\rr\subset\tcal$ be the set of the $k$ root nodes of $\tcal$, each one labelled by an irreducible factor $\psi_{0}$ of $\overline{F}$ (see Figure \ref{figForest}). If we convene that the root nodes have level zero, the \emph{level} of a node $\n\in\tcal\setminus\rr$ is, by definition, the level of its unique previous node plus one. These nodes are labelled by a triple of fundamental invariants, $\n=(\phi_{\n},\lambda_{\n},\psi_{\n})$. \bigskip

\noindent{\bf Notation. }For each $\n\in\tcal$ of level $i$, we denote:\medskip 

$\ty_\n:=$ the type of order $i$ obtained by gathering the fundamental invariants of all nodes in the unique path joining $\n$ with its root node.\medskip
   
$F_\n:=$ the product of all irreducible factors of $F$ which are divisible by $\ty_\n$.\medskip

$\bb_\n:=$ the set of nodes of level $i+1$ whose previous node is $\n$. We say that the nodes of $\bb_\n$ are \emph{branches} of $\n$.
\bigskip

Let $\ll\subset\tcal$ be the set of all leaves of $\tcal$. These leaves are in 1-1 correspondence with the $g$ irreducible factors of $F$ over $\oo_\p$. Suppose that $\n$ is the leaf attached to $F_s$. The level of $\n$ is $r_s+1$, and we denote by $\ty_s:=\ty_\n$ the corresponding type of order $r_s+1$. By construction, $\ty_s$ is an OM representation of $F_s$, and the family of the $\phi_{r_s+1}$ polynomials of $\ty_1,\dots,\ty_g$ is an OM factorization of $F$ over $\oo_\p$. In particular, $F_{\n}=F_s$, by Corollary \ref{faithful2}.

The root nodes are determined by the factorization of $\overline{F}$ over $\F[y]$. Hence, their computation has a cost of $O(n^{2+\epsilon}+n^{1+\epsilon}\log(q))$ $\p$-small operations.
Let 
$$
\mbox{\tt Rout:=}\{{\tt Newton},\, {\tt ResidualPolynomial},\, {\tt Factorization},\, {\tt Representative}\},
$$
be the family of the four fundamental subroutines of the Montes algorithm. For each routine $\mbox{\tt R}\in\mbox{\tt Rout}$ and each node $\nm\in\tcal\setminus\ll$, let $B_{\mbox{\tt R},\nm}$ be an upper bound of the cost, measured in number of $\p$-small operations, of any call to {\tt R} along the different iterations of the WHILE loop that are necessary to compute all nodes of $\bb_\nm$. Then, the total cost of the Montes algorithm is 
\begin{align}\label{wholecomplexity}
O\left(n^{2+\epsilon}+n^{1+\epsilon}\log(q)+\sum\nolimits_{\mbox{\tt R}\in\mbox{\tt Rout}}\sum\nolimits_{\nm \in\tcal\setminus\ll}B_{\mbox{\tt R},\nm}\right).
\end{align}
Our first task is to find estimations for these upper bounds $B_{\mbox{\tt R},\nm}$.

\begin{lemma}\label{branches}
For all $\nm\in\tcal\setminus\ll$, we have 
$F_{\nm}=\prod\nolimits_{\n\in\bb_\nm}F_\n$.
\end{lemma}

\begin{proof}
For an arbitrary node $\n\in\tcal$, let $\ll_\n\subset\ll$ be the set of leaves that are connected to $\n$. By definition, $F_\n$ is the product of all irreducible factors of $F$ attached to the leaves in $\ll_\n$. On the other hand, $\ll_{\nm}$ is clearly the disjoint union of all $\ll_\n$, for $\n\in\bb_\nm$. 
\end{proof}

\begin{lemma}\label{BR}
Let $\nm\in\tcal\setminus\ll$ be a node of level $i-1\ge0$. Let $e_j,f_j,h_j$, $0\le j<i$ be the Okutsu invariants of the type $\ty_{\nm}$, and take $m_i:=e_{i-1}f_{i-1}m_{i-1}$. Denote
$$
B:=\sum\nolimits_{\n\in\bb_\nm\setminus\ll}|\lambda_{\n}|\,\dfrac{\deg F_\n}{m_i}+\sum\nolimits_{\n\in\bb_\nm\cap\ll}\dfrac{v_\p(\res(F_\n,F_t))}{f_0\cdots f_{i-1}},
$$
where, for each $\n\in\bb_\nm\cap\ll$, $F_t\ne F_\n$ is an adequate choice of an irreducible factor of $F$ such that $\ty_{\nm}\mid F_t$.
Then, for {\tt R=Newton} or {\tt Representative}, we have $B_{\mbox{\tt R},\nm}=O\left(n^{1+\epsilon}\delta^{1+\epsilon}B\right)$, whereas for {\tt R=ResidualPo\-lynomial} or {\tt Factorization}, we have
$B_{\mbox{\tt R},\nm}=O\left(n^{1+\epsilon}
f_0\cdots f_{i-1}\log(q)B\right)$.
\end{lemma}

\begin{proof}
Denote for simplicity $\ty=\ty_{\nm}$, $\bb=\bb_\nm$. Since $\nm$ is not a leaf, the type $\ty$ is strongly optimal. Along the construction of the node $\nm$, the algorithm computes an initial representative $\phi$ of $\ty$ (of degree $m_i$) and the positive integer $\omega:=\ord_\ty(F)$. By the definition of $F_{\nm}$, and Lemmas \ref{oftype}, \ref{branches}:
\begin{equation}\label{omega}
\omega=\ord_\ty(F)=\ord_\ty(F_{\nm})=\deg F_{\nm}/m_i=\left(\sum\nolimits_{\n\in\bb}\deg F_\n\right)/m_i.
\end{equation}

Suppose {\tt R = Newton}. In the first iteration of the WHILE loop concerning $\nm$, the routine {\tt Newton($\ty$,$\omega$,$F$)} is called to compute the polygon $N_{i,\omega}(F)$ determined by the first $\omega+1$ coefficients of the $\phi$-expansion of $F$. By Lemma \ref{Newton}, this has a cost of $O(\omega n^{1+\epsilon})$ operations in $A$. By Theorem \ref{precision}, we may work with precision $\delta+1$, so that the computation requires $O(\omega n^{1+\epsilon}\delta^{1+\epsilon})$ $\p$-small operations. By (\ref{omega}), this cost may be distributed into a cost of 
$O((\deg F_\n/m_i) n^{1+\epsilon}\delta^{1+\epsilon})$  $\p$-small operations for each node $\n\in\bb$.

The WHILE loop yields a factorization, $F_{\nm}=\prod\nolimits_{\lambda,\psi}F_{\lambda,\psi}$, where $\lambda$ runs on all slopes of $N_{i,\omega}(F)$ and, for each $\lambda$, the polynomial $\psi$ runs on the monic irreducible factors of $R_{\lambda,i}(F)$. For each ``branch" $(\lambda,\psi)$, a representative $\phi_{\lambda,\psi}$ of the type $\ty_{\lambda,\psi}:=(\ty;(\phi,\lambda,\psi))$ is computed, and the positive integer $\omega_{\lambda,\psi}:=\ord_{\ty_{\lambda,\psi}}(F)$ is determined. The polynomial $F_{\lambda,\psi}$ is, by definition, the product of all irreducible factors of $F$ that are divisible by $\ty_{\lambda,\psi}$. The factorization of $F_{\nm}$ determines in turn a partition, $\bb=\coprod_{\lambda,\psi}\bb_{\lambda,\psi}$, where $\bb_{\lambda,\psi}$ contains all nodes $\n\in\bb$ such that $\ty_{\lambda,\psi}\mid F_\n$.
If $e_\lambda$ is the least positive denominator of $\lambda$ and $f_\psi= \deg \psi$, we have
\begin{equation}\label{omega2}
\deg \phi_{\lambda,\psi}=e_\lambda f_\psi m_i,\quad \omega=\sum\nolimits_{\lambda,\psi}e_\lambda f_\psi\omega_{\lambda,\psi}.
\end{equation}

In order to analyze these branches, there are four different situations to consider.\medskip

\noindent{\bf (a) $\mathbf{\lambda=-\infty}$. }Then, $\bb_{\lambda,\psi}=\{\n\}$ has a single node, which is a leaf of $\tcal$. The irreducible factor attached to this leaf is $F_s=\phi$, and we take $\n=(\phi,-\infty,\hbox{---})$.\medskip

\noindent{\bf (b) $\mathbf{\omega=1}$. }There is only one branch $(\lambda,\psi)$, with $e_\lambda=f_\psi=\omega_{\lambda,\psi}=1$. The set $\bb_{\lambda,\psi}=\{\n\}$ has a single node, which is a leaf of $\tcal$, and we take $\n=(\phi,\lambda,\psi)$.
\medskip
     
\noindent{\bf (c) $\mathbf{e_\lambda f_\psi >1}$. }Then, $\n:=(\phi,\lambda,\psi)\in\bb$ is already a node of level $i$ of $\tcal\setminus\ll$. In other words, $\bb_{\lambda,\psi}=\{\n\}$ singles out already a node of $\bb$, which is not a leaf of $\tcal$.\medskip

\noindent{\bf (d) $\mathbf{\omega>1}$, $\mathbf{e_\lambda f_\psi=1}$. }We fall in a refinement step; the slope $\lambda$ is a negative integer ($e_\lambda=1$), and $\psi$ has degree $f_\psi=1$. We consider $\phi_{\lambda,\psi}$ as a new representative of $\ty$, and $\omega_{\lambda,\psi}$ as the new future length of the Newton polygons of $i$-th order to analyze.\medskip 

In case (d), we take $(\ty,\phi_{\lambda,\psi},\omega_{\lambda,\psi})$ as the input data of a future call of the WHILE loop, yielding a further factorization of $F_{\lambda,\psi}$ and a further partition of $\bb_{\lambda,\psi}$. This loop will follow the same pattern as above, with a minor difference. In the first iteration, $N_{i,\omega}(F)=N_i^-(F)$ is the principal Newton polygon of $F$ with respect to $(\ty,\phi)$; however, after a refinement step, $N_{i,\omega_{\lambda,\psi}}(F)$ is only the part of $N_i^-(F)$ (now with respect to ($\ty,\phi_{\lambda,\psi}$)), formed by the sides of slope greater than $|\lambda|$ in absolute size \cite[Sec. 3]{algorithm}. In any case, the cost of the new call to {\tt Newton} is again $O(\omega_{\lambda,\psi} n^{1+\epsilon}\delta^{1+\epsilon})$ $\p$-small operations, and it may be distributed again into a cost of $O((\deg F_\n/m_i) n^{1+\epsilon}\delta^{1+\epsilon})$ $\p$-small operations for each node $\n\in\bb_{\lambda,\psi}$.    

Therefore, the total cost of the computation of $\bb$ is obtained by counting a cost of $O((\deg F_\n/m_i) n^{1+\epsilon}\delta^{1+\epsilon})$, for each $\n\in\bb$ and for each iteration of the WHILE loop where this node was concerned (i.e. $\n\in\bb_{\lambda,\psi}$). Let us find upper bounds for these numbers of iterations. The discussion is different for $\n$ being a leaf or not. Note that if $\n$ is a leaf then $\deg F_\n/m_i=1$.

Suppose that $\n=(\phi_\n,\lambda_\n,\psi_\n)\in\bb$ is not a leaf. Let $F_s$ be one of the irreducible factors of $F_\n$, and $\t_s\in\ks$ a root of $F_s$. Along the different iterations of the WHILE loop where this node is concerned, we consider different representatives $\phi$ of the type $\ty$ such that $v(\phi(\t_s))$ increases strictly (cf. the proof of Lemma \ref{refinement}). By Proposition \ref{compwidth}, the total number of iterations before we reach the node $\n$ is bounded from above by $\lceil |\lambda_\n|\rceil$. 

Suppose now $\n\in\bb\cap\ll$, and let $F_s$ be the irreducible factor attached to this leaf. We may assume that there are at least two iterations of the WHILE loop concerning $\n$. Let $(\ty,\phi,\omega)$ be the input data of the penultimate of these iterations. Since we do not fall in case (b), we have necessarily $\omega>1$. Let $(\lambda,\psi)$ be the branch such that $\n\in\bb_{\lambda,\psi}$. If   
$\#\bb_{\lambda,\psi}>1$, we take $F_t$ to be an irreducible factor of $F_{\lambda,\psi}$ such that $F_t\ne F_s$. If $\bb_{\lambda,\psi}=\{\n\}$, then $F_{\lambda,\psi}=F_\n$, and the formula (\ref{omega}) shows that $\omega_{\lambda,\psi}=\deg F_{\lambda,\psi}/m_i=1$. By (\ref{omega2}), there is some branch $(\lambda',\psi')\ne(\lambda,\psi)$,  because $\omega>1$ and $e_\lambda=f_\psi=1$; in this case we take $F_t$ to be one of the irreducible factors of $F_{\lambda',\psi'}$. Lemma \ref{resultant} shows in any case that
$$
v(\res(F_s,F_t))/(f_0\cdots f_{i-1})\ge\ell(F_s)\ell(F_t)(V_i+\min\{|\lambda|,|\lambda'|\})\ge \min\{|\lambda|,|\lambda'|\},
$$
where $\ell(F_s),\ell(F_t)$ are the lengths of $N_i(F_s), N_i(F_t)$, respectively. 
In all previous iterations of WHILE, the branch concerning $\n$ was a refinement step, and the absolute size of the corresponding slope was an integer that grows strictly in each iteration; thus, the total number of iterations concerning $\n$ is bounded from above by $1+|\mu|$, for every slope $\mu$ of the Newton polygon of the penultimate iteration.

Therefore, all estimations of the lemma about the contributions of the different nodes $\n\in\bb$ to the total cost of {\tt Newton} are correct. This ends the proof of the lemma in the case  {\tt R=Newton}.

Assume now  {\tt R$\ne$Newton}. In every iteration of the WHILE
 loop, with input data $(\ty,\phi,\omega)$, we compute the residual polynomials
$R_{\lambda,i}(F)$, for $\lambda$ running on all slopes of $N_{i,\omega}(F)$. Then we 
factorize these polynomials over $\F_i$, and for each monic irreducible factor $\psi$ of $R_{\lambda,i}(F)$, we compute a representative of the type $\ty_{\lambda,\psi}$. 

Let $\ell(\lambda)$, $d(\lambda)$ be the length and degree of the side of slope $\lambda$. Lemma \ref{ResidualPolynomial} shows that the cost of the computation of $R_{\lambda,i}(F)$ is $O\left(d(\lambda)(f_0\cdots f_{i-1})(m_i)^{1+\epsilon}\log(q)\right)$ $\p$-small operations. Since $\omega$ is the length of $N_{i,\omega}(F)$, we have 
$$
\omega=\sum\nolimits_{\lambda}\ell(\lambda)=\sum\nolimits_{\lambda}e_\lambda d(\lambda)\ge \sum\nolimits_{\lambda}d(\lambda).
$$ 
Therefore, the total cost of all calls to {\tt ResidualPolynomial} during this iteration is bounded from above by $O\left(\omega(f_0\cdots f_{i-1})(m_i)^{1+\epsilon}\log(q)\right)$. As in the case {\tt R=Newton}, this cost is the product of a constant part, $(f_0\cdots f_{i-1})(m_i)^{1+\epsilon}\log(q)$, times a variable part, $\omega$. As before, we can distribute $\omega$ into a cost of $\deg F_\n/m_i$, for every node of $\bb$, and the same arguments lead to an analogous estimation for $B_{\mbox{\tt R},\m}$, for {\tt R=ResidualPolynomial}, just by changing the constant part.   

Assume now  {\tt R=Factorization}. By Lemma \ref{factorization}, the cost of the factorization of $R_{\lambda,i}(F)$ over $\F_i$ is $O\left(d(\lambda)^{2+\epsilon}(f_0\cdots f_{i-1})^{1+\epsilon}+d(\lambda)^{1+\epsilon}(f_0\cdots f_{i-1})^{2+\epsilon}\log(q)\right)$  $\p$-small operations.
Since $d(\lambda)\le \omega=\deg F_\m/m_i\le n/m_i\le n/(f_0\cdots f_{i-1})$, this cost is $O\left(d(\lambda)n^{1+\epsilon}f_0\cdots f_{i-1}\log(q)\right)$.
Thus, the cost of all calls to {\tt Factorization} during this iteration is $O\left(\omega n^{1+\epsilon}f_0\cdots f_{i-1}\log(q)\right)$.
We obtain the estimation of $B_{\mbox{\tt R},\nm}$ by the same arguments of the previous cases.

Finally, let {\tt R=Representative}. Let $V_{\lambda,\psi}:=(e_\lambda)^2 f_\psi(V_i+|\lambda|)$. By Lemma \ref{Representative}, the cost of the computation of a representative of $\ty_{\lambda,\psi}$ is 
\begin{equation}\label{cost}
O\left((f_0\cdots f_{i-1}f_\psi)^{2+\epsilon}(V_{\lambda,\psi})^{1+\epsilon}\right) \mbox{ $\p$-small operations}.
\end{equation}
Instead of distributing this cost, we now attach the whole cost (\ref{cost}) to every node $\n\in\bb_{\lambda,\psi}$, so that our estimation is sharp only when $\#\bb_{\lambda,\psi}=1$. 

Let us estimate the accumulated cost of every node $\n\in\bb$. Along all refinement steps, we have $f_\psi=1$ and $V_{\lambda,\psi}=V_i+|\lambda|$, where $|\lambda|$ is a positive integer that grows strictly at each iteration; thus, the higher cost of (\ref{cost}) occurs at the last iteration. 

Suppose $\n\in\bb\setminus\ll$. After eventually some refinement steps, in the last ite\-ration, $f_\psi=f_{i,\n}$, $V_{\lambda,\psi}=V_{i+1,\n}$, are Okutsu data of the type $\ty_\n$. Let $F_s$ be any irreducible factor of $F_\n$. As in the proof of Theorem \ref{IrrEstimation}, $ f_0\cdots f_{i-1}f_{i,\n}V_{i+1,\n}\le 2\delta(F_s)\le 2\delta$.
Since there are at most $\lceil |\lambda_{\n}|\rceil$ iterations (Proposition \ref{compwidth}), the accumulated cost of the computation of $\n$ is bounded from above by
$$
\lceil |\lambda_{\n}|\rceil (f_0\cdots f_{i-1}f_{i,\n})^{2+\epsilon}(V_{i+1,\n})^{1+\epsilon}=O\left( |\lambda_{\n}| n\delta^{1+\epsilon}\right).
$$

Finally, let $\n\in\bb\cap\ll$. In the last iteration there is no call to {\tt Representative}. Let $(\ty,\phi,\omega)$ be the input data of the penultimate iteration, and let $(\lambda,\psi)$ be the branch concerning $\n$. Let $u$ be the ordinate of the left end point of the side of slope $\lambda$ of $N_{i,\omega}(F)$. Since $u\ne0$ and we work with precision $\delta+1$, we have necessarily $u\le \delta e_0\cdots e_{i-1}$. Now, $V_i+|\lambda|$ is the ordinate of the left end point of $N_i(F_\n)$; by the theorem of the product, $V_i+|\lambda|\le u\le \delta e_0\cdots e_{i-1}$.  
As we saw along the proof of the case {\tt R=Newton}, the total number of all-but-last iterations is bounded from above by $v(\res(F_\n,F_t))/(f_0\cdots f_{i-1})$; thus, the accumulated cost of all calls to {\tt Representative} along the computation of $\n$ is 
$$
O\left((f_0\cdots f_{i-1})^{1+\epsilon}(V_i+|\lambda|)^{1+\epsilon}v(\res(F_\n,F_t))\right)=O\left((m_i)^{1+\epsilon}\delta^{1+\epsilon}v(\res(F_\n,F_t))\right).
$$
This ends the proof of the lemma.
\end{proof}

\begin{theorem}\label{main}
The cost of the Montes algorithm over $\oo_\p$, applied to a monic se\-pa\-rable polynomial $F\in A[x]$ of degree $n$ is
$O\left(n^{2+\epsilon}+n^{1+\epsilon}(1+\delta)\log(q)+n^{1+\epsilon}\delta^{2+\epsilon}\right)$ $\p$-small operations, where $\delta:=v_\p(\dsc(F))$.
\end{theorem}

\begin{proof}
Let $\nn:=\tcal\setminus(\rr\cup\ll)$ be the set of nodes that are neither a root nor a leaf of $\tcal$. Let us denote 
 $\lambda_{i,s},e_{i,s},f_{i,s}, m_{i,s}$, etc. for the Okutsu invariants of $\ty_s$ at level $i\le r_s$. Also, we denote $\rho_{s,t}:=v(\res(F_s,F_t))$, for all $1\le s\ne t\le g$. 

We shall use the estimation (\ref{sumdelta}), and two obvious identities:
\begin{equation}\label{obvious}
\sum_{\n\in\nn}|\lambda_{\n}|\dfrac{\deg F_\n}{m_\n}=
\sum_{1\le s\le g}\sum_{i=1}^{r_s} |\lambda_{i,s}| \dfrac{n_s}{m_{i,s}},\quad
\sum_{1\le s\le g}\left(\delta_s+\rho_{s,t}\right)=O(\delta).
\end{equation}
By (\ref{wholecomplexity}), we need only to estimate $\sum\nolimits_{\nm \in\tcal\setminus\ll}B_{\mbox{\tt R},\nm}$, for each subroutine {\tt R$\in$Rout}.\medskip

\noindent{\tt R=Newton} or {\tt Representative}. By Lemma \ref{BR}, (\ref{sumdelta}) and (\ref{obvious}),
\begin{align*}
\sum\nolimits_{\nm \in\tcal\setminus\ll}B_{\mbox{\tt R},\nm}&\ \le n^{1+\epsilon}\delta^{1+\epsilon}\left(\sum\nolimits_{\n\in\nn}|\lambda_{\n}|\dfrac{\deg F_\n}{m_\n}+\sum\nolimits_{\n\in\ll}\rho_{s,t}\right)\\
&\ =n^{1+\epsilon}\delta^{1+\epsilon}\left(\sum\nolimits_{1\le s\le g}\left(\sum\nolimits_{1\le i\le r_s} |\lambda_{i,s}| \dfrac{n_s}{m_{i,s}}\right)+\rho_{s,t}\right)\\
&\ =n^{1+\epsilon}\delta^{1+\epsilon}O\left(\sum\nolimits_{1\le s\le g}\delta_s+\rho_{s,t}\right)=O\left(n^{1+\epsilon}\delta^{2+\epsilon}\right).
\end{align*}

\noindent{\tt R=ResidualPolynomial} or {\tt Factorization}. The argument is analogous.
\begin{align*}
\sum_{\nm \in\tcal\setminus\ll}B_{\mbox{\tt R},\nm}&\ \le n^{1+\epsilon}\log(q)\left(\sum\nolimits_{\n\in\nn}|\lambda_{\n}|f_0\cdots f_{i-1}\dfrac{\deg F_\n}{m_\n}+\sum\nolimits_{\n\in\ll}\rho_{s,t}\right)\\
&\ =n^{1+\epsilon}\log(q)\left(\sum\nolimits_{1\le s\le g}\left(\sum\nolimits_{1\le i\le r_s} |\lambda_{i,s}| \dfrac{f_{0,s}\cdots f_{i-1,s}n_s}{m_{i,s}}\right)+\rho_{s,t}\right)\\
&\ =n^{1+\epsilon}\log(q)O\left(\sum\nolimits_{1\le s\le g}\delta_s+\rho_{s,t}\right)=O\left(n^{1+\epsilon}\delta\log(q)\right).
\end{align*}
\end{proof}

\begin{corollary}\label{psmall}
The complexity of the Montes algorithm is $O\left(n^{2+\epsilon}+n^{1+\epsilon}\delta^{2+\epsilon}\right)$ word operations, if $\p$ is small. 
\end{corollary}

\subsection{Approximate factorization of polynomials over local fields}\label{subsecAppFactorization}
Theorem \ref{main} leads to an improvement of the complexity estimates of all routines mentioned in the Introduction. In this section, we discuss the new estimation obtained for the factorization of polynomials over local fields, up to a prescribed precision.

Let  $F\in A[x]$ be a monic separable polynomial of degree $n$, and denote $\delta:=v_\p(\dsc(F))$. Let $\p$ be a non-zero prime ideal of $A$, and $F_1,\dots,F_g\in\oo_\p[x]$ the irreducible factors of $F$ over $\oo_\p$. 
Suppose an OM factorization of $F$ over $\oo_\p[x]$ has been computed, in the form of a family $\ty_{F_1},\dots,\ty_{F_g}$ of OM representations of the irreducible factors, that faithfully represents $F$, and satisfies (\ref{distinguishOM}).
Then, the single-factor lifting algorithm (SFL) derives from each $\ty_{F_s}$ a monic polynomial $P_s\in A[x]$, irreducible over $\oo_\p$, such that $P_s\approx F_s$ and $P_s\equiv F_s\md{\m^\nu}$, for an arbitrary prescribed precision $\nu$. 

\begin{theorem}\label{SFL}
The SFL algorithm requires $O(nn_s\nu^{1+\epsilon}+n\delta_s^{1+\epsilon})$ $\p$-small operations, where $n_s:=\deg F_s$, $\delta_s:=\delta(F_s)$.
\end{theorem}

\begin{proof}
Let $r_s$ be the Okutsu depth of $F_s$. Along the proof of \cite[Lem. 6.5]{GNP}, it is obtained an estimation of $O\left(nn_s(\nu^{1+\epsilon}+(V_{r_s+1}/e(F_s))^{1+\epsilon})\right)$ $\p$-small operations. In Lemma \ref{delta0props} we have seen that the Okutsu discriminant $\delta_0(F_s):=V_{r_s+1}/e(F_s)$ is bounded from above by
$2\delta_s/n_s$. This proves the theorem.
\end{proof}

By applying the SFL routine to each OM representation $\ty_{F_1},\dots,\ty_{F_g}$, we get an 
OM factorization, $F\approx P_1\cdots P_g$, such  that $P_s\equiv F_s\md{\m^\nu}$, for all $1\le s\le g$.

\begin{theorem}\label{factorization}
A combined application of the Montes and SFL algorithms, computes an OM factorization of $F$  with prescribed precision $\nu$, at the cost of 
$$O\left(n^{2+\epsilon}+n^{1+\epsilon}(1+\delta)\log q+n^{1+\epsilon}\delta^{2+\epsilon}+n^2\nu^{1+\epsilon}\right)\ \mbox{ $\p$-small operations}.$$
 
If $\p$ is small, we obtain a cost of  
$O\left(n^{2+\epsilon}+n^{1+\epsilon}\delta^{2+\epsilon}+n^2\nu^{1+\epsilon}\right)$ word operations.
\end{theorem}

\begin{proof}
The estimation is obtained by adding to the cost of the Montes algorithm, given in Theorem \ref{main}, the sum of the costs of SFL given in Theorem \ref{SFL}, for $1\le s\le g$, having in mind that $n_1+\cdots+n_g=n$, $\delta_1+\cdots+\delta_g\le \delta$.  
\end{proof}

In comparison with previous estimations, the total degree in $n$, $\delta$ and $\nu$ is reduced from $4+\epsilon$ to $3+\epsilon$.

\end{document}